\documentclass[a4paper,final]{amsart}
\usepackage[T1]{fontenc}
\usepackage[utf8]{inputenc}
\usepackage{lmodern}
\usepackage[english]{babel}

\newcommand{\highresticks}{10}
\newcommand{\lowresticks}{2}

\usepackage{cutwin}
\usepackage{longtable}

\usepackage{microtype}
\usepackage{ifdraft}
\usepackage{array}
\usepackage{wrapfig}
\usepackage{floatflt}

\usepackage{amsmath}
\usepackage{amsthm}
\usepackage{amsbsy}
\usepackage{amssymb}
\usepackage{amsrefs}
\usepackage{ifthen}

\usepackage{enumitem}
\usepackage{xkeyval}

\usepackage{hhline}
\usepackage{graphicx}
\usepackage{verbatim} 

\usepackage[l2tabu, orthodox]{nag}
\usepackage[pdfusetitle]{hyperref}

\usepackage{tikz-cd}
\usetikzlibrary{calc,fadings,shapes,arrows,decorations.pathreplacing,backgrounds,positioning}

\makeatletter
\newcommand{\manuallabel}[2]{\def\@currentlabel{#2}\label{#1}}
\makeatother

\allowdisplaybreaks[3]

\tikzset{%
   >=latex, 
   inner sep=0pt,%
   outer sep=2pt,%
   mark coordinate/.style={inner sep=0pt,outer sep=0pt,minimum size=3pt,
     fill=black,circle}%
}

\usepackage{blindtext}

\usepackage{parskip}
\usepackage{ragged2e}
\RaggedRight

\usepackage{caption}
\ifdraft{\usepackage[notref,notcite]{showkeys}}{}

\usepackage{mathtools}
\DeclarePairedDelimiterX\scprod[2]{\langle}{\rangle}{#1,#2}   
\DeclarePairedDelimiter\norm{\vert\mkern-1.7mu\vert}{\vert\mkern-1.7mu\vert}      
\DeclarePairedDelimiterX\setof[2]{\lbrace\,}{\,\rbrace}{#1\,\delimsize\vert\,#2}
\mathtoolsset{
showonlyrefs=true 
}

\theoremstyle{plain}
\newtheorem{thm}{Theorem}
\newtheorem{lemma}[thm]{Lemma}

\newtheorem{coro}[thm]{Corollary}
\theoremstyle{definition}
\newtheorem{example}[thm]{Example}
\newtheorem{definition}[thm]{Definition}

\numberwithin{thm}{section}

\newcommand{\os}{^\ast}
\newcommand{\oriented}{^{or}}

\newcommand{\card}[1]{\vert #1\vert}
\newcommand{\celldim}[1]{\operatorname{cd}_{#1}}
\newcommand{\acell}{{\mathfrak A}}
\newcommand{\rcell}{{\mathfrak R}}
\newcommand{\bcell}{{\mathfrak B}}
\newcommand{\xcell}{{\mathfrak X}}
\newcommand{\RR}{{\mathbb R}}
\newcommand{\inversions}[1]{{\operatorname{inv}(#1)}}
\newcommand{\qinversions}[1]{{\widetilde{\operatorname{inv}}(#1)}}

\newcommand{\qrises}[1]{{\widetilde{\operatorname{rises}}(#1)}}
\DeclareMathOperator{\Sl}{Sl}
\DeclareMathOperator{\diag}{diag}
\DeclareMathOperator{\id}{id}
\DeclareMathOperator{\Tr}{Tr}
\DeclareMathOperator{\vspan}{span}

\DeclareMathOperator{\Spec}{Spec}
\DeclareMathOperator{\Grassm}{Grassm}
\DeclareMathOperator{\Invol}{Invol}

\newcommand{\milnor}[1]{\cite{MR0440554}*{#1}}

\newcommand{\bigslant}[2]{{\raisebox{.3em}{$#1$}\left/\raisebox{-.3em}{$#2$}\right.}}

\newcommand{\lie}[1]{{\mathfrak{#1}}}

\DeclareMathOperator{\Gl}{Gl}

\DeclareMathOperator{\sign}{sign}

\newcommand{\rotat}[2]{{\mathrm{Rot}}^{#1}_{#2}}

\newcommand{\incid}[2]{J(#1,#2)}
\newcommand{\orid}[2]{o(#1,#2)}
\newcommand{\coveredby}{\mathrel{\triangleleft}}
\newcommand{\incitti}[1]{\cite{MR2106960}*{#1}}

\newcommand{\covtrans}{\mathop{\mathrm{ct}}\nolimits}
\newcommand{\covop}{\mathop{\mathrm{co}}\nolimits}

\newcommand{\involutions}{\mathop{\mathrm{Invol}}\nolimits}
\newcommand{\wtp}[1]{\mathop{\mathrm{tp}}\nolimits_w(#1)}

\begin{document}

\title{The Bruhat decomposition of real Grassmann manifolds}
\author{Christian Nassau}
\email{nassau@nullhomotopie.de}
\date{\today}

\begin{abstract}
We study the Grassmann manifold $G_k$ of all $k$-dimensional subspaces of $\RR^n$.
The Cartan embedding $G_k\subset O(n)$ realizes $G_k$ as a subspace of $\Sl_n(\RR)$
and we study the decomposition $G_k=\coprod_w (BwB\cap G_k)$ inherited from the classical Bruhat decomposition.
We prove that this defines a CW structure on $G_k$
and determine the incidence numbers between cells.
\end{abstract}

\maketitle

\tableofcontents


\begin{section}{Introduction}
The study of Grassmann manifolds has a long history
and their topology is exceptionally well understood.
The theory of Schubert cells provides the Grassmannian $G_k(\RR^n)$
with a CW-structure, and the incidence relations between the cells
have long been known (see \cite{jungkind}, for example, for a thorough discussion of the real case).
The reader may thus be forgiven to question the necessity of another cell decomposition.

The situation is somewhat different for quotients of Grassmann manifolds. Consider
the space $PG_k(\RR^{2k}) = G_k(\RR^{2k})/{\mathbin\perp}$
where a subspace $V$ is identified with its complement $V^\perp$.
Schubert cells do not provide a CW decomposition of this quotient, since
the pointwise complement of a Schubert cell $X_w$  based on a flag $F_\bullet$
is a Schubert cell with respect to the different flag $F_{n-\bullet}^\perp$.
In fact, the author believes that the cohomology of the $PG_k$ (and
their oriented counterparts $PG_{2k}\oriented$) has not yet been
documented in the literature.

The Bruhat decomposition that we study in this paper, in contrast, 
is compatible with complements. 
The decomposition has shown up in many places before, but it seems it has
not acquired a standard name, nor have its combinatorial properties been investigated in detail. We can
think of $G_k(\RR^n)$ as a subset of the orthogonal group $O(n)$
by mapping a $V_k\subset\RR^n$ to the reflection $r_V$ at $V$
(Cartan embedding). Let $B^+\subset\Sl_n$ be the Borel group of upper triangular matrices
with positive diagonal. 
We have
$$G_k(\RR^n) = \coprod_{\tilde w} \left( G_k(\RR^n) \cap B^+\tilde wB^+ \right)$$
where $\tilde w$ runs through the set of signed permutation matrices
(actually, only involutions need to be considered).
We call this the {\em Bruhat decomposition} and $\bcell_{\tilde w} = G_k(\RR^n) \cap B^+\tilde wB^+$ a {\em Bruhat cell}.
The reader can find a picture for the case of $\RR P^2$ on page \pageref{fig:rp2example}.

Our study begins in section \ref{sec:cell}
where we will describe a natural coordinate system on each $\bcell_{\tilde w}$.
This will in particular show that each Bruhat cell is indeed a topological cell.

Examples show that these coordinates cannot be extended continuously to the closed cells, however.
Section \ref{sec:attaching} will therefore be dedicated to the construction of
alternative attaching maps which exhibit the $\bcell_{\tilde w}$ as the cells of a CW-decomposition.

In section \ref{sec:incidence} we begin the study of the combinatorial relations
between the $\bcell_{\tilde w}$. 
Our investigation is based on the work of Incitti \cite{MR2106960}
who has already solved the underlying combinatorial problem for unsigned permutations.
Lifting his results into the geometric context involves a rather detailed
study of (an analogue of) Richardson varieties that are well-known, for example from
the literature on Kazhdah-Lusztig varieties.
The upshot is that the incidence numbers between the Bruhat cells are all either $0$, $1$ or $-1$.
The precise results are 
stated in Theorem \ref{l:signdet} and Lemmas \ref{comp:incidmodel}, \ref{comp:incid}, \ref{comp:ormodel}, and
\ref{comp:orid}.

Our proofs (in particular in section \ref{sec:incidence}) are unfortunately quite
computational. 
This paper wouldn't exist without the assistance of 
the Sage computer algebra system \cite{sage} 
(using the Singular \cite{singular} and GAP \cite{GAP4}
libraries under the surface).
We hope to include some of our code in a future revision, once we find the time to clean it up.

\end{section}


\begin{section}{Cell decomposition}\label{sec:cell}
Let $G_k=G_k(\RR^n)$ be the set of $k$-dimensional vector subspaces of $\RR^n$. We have an embedding
$r:G_k\rightarrow O(n)$ that maps a space $V$ to the reflection across $V$.
We will identify $G_k$ with its image
$$r(G_k) = \setof{g\in O(n)}{g^2=\id,\,\Tr g=n-2k}$$
and think of $G_k$ as the space of orthogonal pairs $(V^+,V^-)$ where $V^\pm$ denotes the
$(\pm1)$-eigenspace of $r(V)$.

Let $B\subset\Sl_n(R)$ be the set of upper triangular matrices
and $B^+\subset B$ the subset of matrices with positive diagonal.
Let $W$ be the set of permutation matrices, and $\tilde W = \diag(\pm 1)\cdot W$ the set of signed permutation matrices. 
We then have the well-known Bruhat decompositions
\begin{equation}
\Sl_n(R) = \coprod_{w\in W} BwB = \coprod_{\tilde w\in \tilde W} B^+\tilde wB^+.
\end{equation}

The Grassmannian inherits an induced decomposition
\begin{align}
G_k &= \coprod_{\tilde w\in \tilde W} \left( G_k \cap B^+\tilde wB^+ \right).
\end{align}
We call $\bcell_{\tilde w} = G_k \cap B^+\tilde wB^+$
the \emph{Bruhat-cell} associated to the signed permutation $\tilde w$.

\begin{lemma}
$\bcell_{\tilde w}$ is empty unless $\tilde w\tilde w=1$.
\end{lemma}
\begin{proof}
Since $g=g^{-1}$ for every $g\in r(G_k)$ this follows from
$(B^+\tilde wB^+)^{-1} = B^+{\tilde w^{-1}}B^+$.
\end{proof}

Conversely, if $\tilde w\tilde w=1$ then $\tilde w$ is itself an element of $\bcell_{\tilde w}$,
so all such $\bcell_{\tilde w}$ are non-empty. 
We refer to $\tilde w$ as the \emph{center} of $\bcell_{\tilde w}$.

Henceforth $\tilde w\tilde w=1$ will be implicitly assumed.
For such a signed involution $\tilde w$ we let $w$ denote the underlying permutation
and $\epsilon = \diag(\epsilon_i)$ the associated diagonal matrix of signs $\epsilon_i\in\{\pm1\}$. 
We have $\epsilon_i=\epsilon_{w(i)}$ because $\tilde w^2=1$.
Let $e_i$ denote
the standard basis vectors of $\RR^n$. 
We have $\tilde w(e_i) = \epsilon_i e_{w(i)}$.

\begin{lemma}\label{l:alphadef}
$\bcell_{\tilde w} = \setof{\alpha \tilde w\alpha^{-1}}{\alpha\in B^+} \cap O(n)$
\end{lemma}
\begin{proof}
Suppose $g\in \bcell_{\tilde w}$. We have to show that there is an $\alpha\in B^+$ such that
$g=\alpha \tilde w \alpha^{-1}$.
The following line of reasoning seems to be well-known. Our proof is modelled on \cite{MR2278145}*{proof of Theorem 2}.

Let $U=\setof{u\in B^+}{u_{i,i}=1}$ be the unipotent subgroup and
define 
\begin{equation}
U^+ = U\mathbin{\cap} wUw,\quad  U^- = U\mathbin{\cap} wU\os w.
\end{equation}
Here $U\os$ consists of the transposed matrices from $U$.
Every $x\in U$ has a unique factorization $x=(x_-)\cdot(x_+)$ with $x_\pm\in U^\pm$.

The Bruhat decomposition can be refined to $g=u \tilde wD v$ with $u,v\in U$ and a
diagonal matrix $D\in B^+$ with positive entries. Furthermore, this decomposition becomes unique under
the additional assumption that $u\in U^-$.
 
Consider now
\begin{align}
(\tilde wD)^{-1} &= vg^{-1}u = vgu = (vu) (\tilde wD) (vu) 
\\ &= (vu)_-(vu)_+(\tilde wD)(vu) 
= (vu)_- \cdot (\tilde wD) \cdot \left((vu)_+\right)^{(\tilde wD)} (vu)
\end{align}
Both sides of this equation are Bruhat decompositions that satisfy the uniqueness condition. This allows us to conclude
\begin{equation}
(\tilde wD)^{-1} = \tilde wD,\quad
(vu)_- = 1,\quad
(vu)^{-1} = \left((vu)_+\right)^{(\tilde wD)}.
\end{equation} 
We therefore have $D^{\tilde w}=D^{-1}$, $(vu) = (vu)_+$ and $(vu)\in U^+\cap (\tilde wD)^{-1}U^+(\tilde wD)$.
Since this group is unipotent there is a unique square root $L$ with $L^2 = vu$,
given, for example, by the binomial formula.
By the uniqueness of the square root we also have $L^{-1} = L^{(\tilde wD)}$.

Now let $\alpha = uL^{-1}D^{(-1/2)}$. We then find
\begin{align}
\alpha \tilde w \alpha^{-1} &= uL^{-1}D^{(-1/2)}\tilde wD^{(1/2)} L u^{-1}
= uL^{-1}\tilde w D L u^{-1}
= u \tilde w D L^2 u^{-1}
= u \tilde w D v
\end{align}
as desired.
\end{proof}

For the next Lemma let 
$H_i = \setof{(x_1,\ldots,x_i,0,\ldots,0)}{x_i>0}$.
\begin{lemma}\label{l:normform}
Let $V^\pm$ be the $(\pm1)$-eigenspaces of $g\in\bcell_{\tilde w}$.
Then there are $v_i\in H_i$ such that
\begin{equation}\label{eq:normform}
V^\pm = \vspan\setof{v_i \pm \epsilon_i v_{w(i)}}{i\le w(i)}.
\end{equation}
Conversely, for all such $(v_i)$ the corresponding pair $(V^+,V^-)$ defines an element of $\bcell_{\tilde w}$
as long as $V^+$ and $V^-$ are orthogonal.
\end{lemma}
\begin{proof}
This is just a restatement of the previous Lemma. Given $\alpha\in B^+$ with $g=\alpha\tilde w\alpha^{-1}$ 
one has $V^\pm = \alpha(W^\pm)$ where 
\begin{equation}
W^\pm = \vspan\setof{e_i \pm \epsilon_i e_{w(i)}}{i\le w(i)}
\end{equation}
are the $(\pm1)$-eigenspaces of $\tilde w$. 
Formula \eqref{eq:normform} results on setting $v_i = \alpha(e_i)$.
\end{proof}

We have now seen that a point of $\bcell_{\tilde w}$
can by specified by giving a certain $\alpha\in B^+$ as in Lemma \ref{l:alphadef},
or equivalently the $v_i=\alpha(e_i)$ as in Lemma \ref{l:normform}.
Such an $\alpha$ is itself determined by the matrix $A=\alpha\os\alpha$ with $A_{i,j} = \scprod{v_i}{v_j}$
- the reconstruction of $\alpha$ from $A$ is known as the \emph{Cholesky decomposition} of $A$.
An $A$ that arises in this way has the following properties:
\newcommand{\eqline}[1]{\makebox[.8\textwidth][l]{#1}}
\begin{align}
&\eqline{$A$ is symmetric and positive definite.} \label{A:c1} \\
&\eqline{$A$ is $\tilde w$-invariant: $A^{\tilde w} = A$.} \label{A:c2}
\end{align}
This second condition encodes the fact that the $V^\pm$ are orthogonal to each other:
indeed, the $g=\alpha \tilde w\alpha^{-1}$
associated to $\alpha$ is only orthogonal when
\begin{equation}
1 = gg\os  = \alpha \tilde w \alpha^{-1} (\alpha^{-1})\os \tilde w \alpha\os
\end{equation}
which is equivalent to $A=A^{\tilde w}$.
The next Lemma shows that we can further require
\begin{align}
&\eqline{$A_{i,i} = 1$ and $A_{i,j} = 0$ if $i<j$ and $w(i)<w(j)$.} \label{A:c3}
\end{align}

\begin{lemma}\label{l:uniquenormform}
Every $V^\pm\in\bcell_{\tilde w}$ has a normal form \eqref{eq:normform} 
where $\norm{v_i}=1$ and
$\scprod{v_i}{v_j}=0$ if $i<j$ and $w(i)<w(j)$.
Under these assumptions the $v_i$ are uniquely determined by $V^\pm$.
\end{lemma}
\begin{proof}
We need to investigate the indeterminacy of $\alpha$ in the equation $g=\alpha\tilde w\alpha^{-1}$.
Clearly, $\alpha$ and $\alpha'$ give the same $g$ if and only if $\alpha' = \alpha \beta$ with $\beta\in C(\tilde w)\cap B^+$.
These allowable $\beta$ are explicitly given by the conditions
\begin{align}
\beta_{i,j} &= \begin{cases} \epsilon_i\epsilon_{j}\beta_{w(i),w(j)} & \text{$i\le j$ and $w(i)\le w(j)$,} \\
0 & \text{$i>j$ or $w(i)>w(j)$}.
\end{cases}
\end{align}
For $A=\alpha\os\alpha$ the normalization procedure
replaces $A$ by $\beta\os A\beta$ for such $\beta$.

Assume inductively that for some $k$ the submatrix of $A$ with $i,j<k$ is already normalized.
Let $R_k = \setof{j}{j<k,\, w(j)<w(k)}$ be the set of rises
in the $k$-th column; we need to find an $A'=\beta\os A\beta$ with $A'_{k,j}=0$ for $j\in R_k$.

Note that by $\tilde w$-invariance we have $A_{k,j} = \pm A_{w(k),w(j)}$, so
if $w(k)<k$ the $A_{k,j}$ with $j\in R_k$ are already zero. We can therefore assume
$w(k)\ge k$.

Let $B=(A_{i,j})_{i,j\in R_k}$ be the submatrix corresponding to the rises.
Since $A$ is positive definite, the same is true of the submatrix $B$. In particular, $B$ is invertible.
Let $e_j=A_{k,j}$ for $j\in R_k$ be the vector of error terms, and define correction terms $c$ via $c=B^{-1}(e)$.
Now let $\beta\in C(\tilde w)$ be the matrix with ones on the diagonal and
\begin{align}
-c_i& &\text{in the $i$-th row of the $k$-th column for $i\in R_k$,} \\
-\epsilon_i\cdot \epsilon_k\cdot c_i & &\text{in the $w(i)$-th row of the $w(k)$-th column for $w(i)\in R_k$.}
\end{align}
One can easily check that $A'=\beta\os A\beta$ now has $A'_{k,j}=0$ for $j\in R_k$, as desired.
A suitable diagonal matrix $\beta'$ will then transform $A'$ into an $A''$ that additionally also obeys
$A''_{k,k}=1$, which completes the inductive step.
\end{proof}

\begin{definition}\label{acelldef}
Let $\acell_{\tilde w}$
be the space of matrices $A$ satisfying the conditions \eqref{A:c1}, \eqref{A:c2}, \eqref{A:c3}.
\end{definition}

For a permutation $w$ let 
$$\inversions{w} = \setof{(i,j)}{i<j,\, w(i)>w(j)}$$
be the set of inversions of $w$. If $w$ is an involution one can also define the quotient
\begin{align}
\qinversions{w} &= \bigslant{\,\inversions{w}}{(i,j)\sim (w(j),w(i))}.
\end{align}
\begin{coro}
$\bcell_{\tilde w}$ is an open cell of dimension $\celldim{\tilde w} = \card{\qinversions{w}}$.
\end{coro}
\begin{proof}
The space $\acell_{\tilde w}$
is easily seen to be an open convex set of the indicated dimension.
By Lemma \ref{l:uniquenormform} it is homeomorphic to $\bcell_{\tilde w}$.
\end{proof}

We refer to $\celldim{\tilde w}$ as the \emph{cell dimension} of $\tilde w$.

We have so far set up a homeomorphism from
$\acell_{\tilde w}$ onto the Bruhat cell $\bcell_{\tilde w}$.
Unfortunately, that map can not in general be extended continuously to the closure
$\overline{\acell_{\tilde w}}$. This can already be seen in the case of $G_1(\RR^3) = \RR P^2$. 
Figure \ref{fig:rp2example} shows its Bruhat decomposition.
The involution 
$\tilde w=\left(\begin{smallmatrix}&&1\\&1&\\1&&\end{smallmatrix}\right)$
corresponds to the $2$-cell $B$. The space $\acell_{\tilde w}$ is given by
\begin{equation}
\acell_{\tilde w} = \left\{\,\left(
\begin{array}{rrr}1&y&x\\y&1&y\\x&y&1\end{array}
\right)\,\middle\vert\,-1+2y^2< x<1,\, \norm{y}< 1\right\}
\end{equation}
The effect of the map $\acell_{\tilde w}\rightarrow \bcell_{\tilde w}$ along the boundary is shown in the following picture:
The entire segment $x=1$, $\norm{y}\le 1$ is mapped to the point $Q$, whereas the endpoints
$(1,\pm 1)$ are blown-up onto the $1$-cells $\delta$ and $\gamma$.
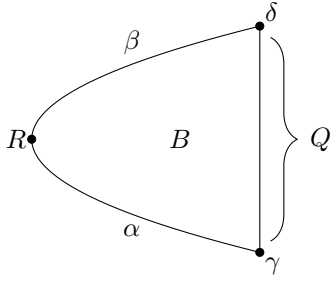
\begin{equation}
\begin{tikzpicture}[scale=1.5]
\draw (1,-1) -- (1,0) -- (1,1);
\draw [decorate,decoration={brace,amplitude=10pt,mirror,raise=4pt},yshift=0pt]
(1,-0.9) -- (1,0.9) node [midway,xshift=0.8cm] {$Q$};
\draw[domain=-1:0] plot (-1+2*\x*\x,\x);
\draw[domain=0:1]  plot (-1+2*\x*\x,\x);
\draw[fill] (-1,0) circle [radius=1pt] node[left] {$R$};
\draw[fill] (1,-1) circle [radius=1pt] node[below right] {$\gamma$};
\draw[fill] (1,+1) circle [radius=1pt] node[above right] {$\delta$};
\draw node at (0.3,0) {$B$};
\draw node[above left] at (0,0.7071) {$\beta$}; 
\draw node[below left] at (0,-0.7071) {$\alpha$}; 
\end{tikzpicture}
\end{equation}
This example shows that we need a different 
coordinate system to turn the $\bcell_{\tilde w}$
into the cells of a CW-decomposition.

\newcommand{\ttmat}[1]{\left(\begin{smallmatrix}#1\end{smallmatrix}\right)}
\begin{figure}
\begin{tikzpicture}[scale=1.5]
\draw (2,2) circle [radius=2];
\draw[thin,->] (4,2) arc (0:45:2) node[above right] {$\alpha$};
\draw[thin,->] (2,4) arc (90:135:2) node[above left] {$\beta$};
\draw[thin,->] (0,2) arc (180:225:2) node[below left] {$\alpha$};
\draw[thin,->] (2,0) arc (270:315:2) node [below right] {$\beta$};
\draw[thin,->] (2,0) -- (2,1) node[right] {$\gamma$}; 
\draw[thin,->] (2,2) -- (2,3) node[right] {$\delta$}; 
\draw (2,0) -- (2,4);
\draw[fill] (2,0) circle [radius=1pt] node[below] {$P$};
\draw[fill] (2,4) circle [radius=1pt] node[above] {$P$};
\draw[fill] (2,2) circle [radius=1pt] node[left]  {$Q$};
\draw[fill] (4,2) circle [radius=1pt] node[right] {$R$};
\draw[fill] (0,2) circle [radius=1pt] node[left]  {$R$};
\draw node at (1,2) {$A$}; 
\draw node at (3,2) {$B$}; 
\end{tikzpicture}
\caption{The Bruhat decomposition of $\RR P^2$. \label{fig:rp2example}
It consists of three $0$-cells $P$, $Q$, $R$, four 1-cells $\alpha$, $\beta$, $\gamma$, $\delta$
and two $2$-cells $A$ and $B$.
The corresponding signed involutions are}
\bigskip
\begin{small}
\begin{centering}
\begin{tabular}{rrr}
$P \simeq \ttmat{+1&&\\&-1&\\&&+1},$ & $Q \simeq \ttmat{-1&&\\&+1&\\&&+1},$ & $R \simeq \ttmat{+1&&\\&+1&\\&&-1},$
\\[10pt]
$\{\alpha,\beta\} \simeq \ttmat{+1&&\\&&\pm1\\&\pm1&}$, &
$\{\gamma,\delta\} \simeq \ttmat{&\pm1&\\\pm1&&\\&&+1}$, &
$\{A,B\} \simeq \ttmat{&&\pm 1\\&+1&\\\pm1&&}$.
\end{tabular}
\end{centering}
\end{small}
\end{figure}
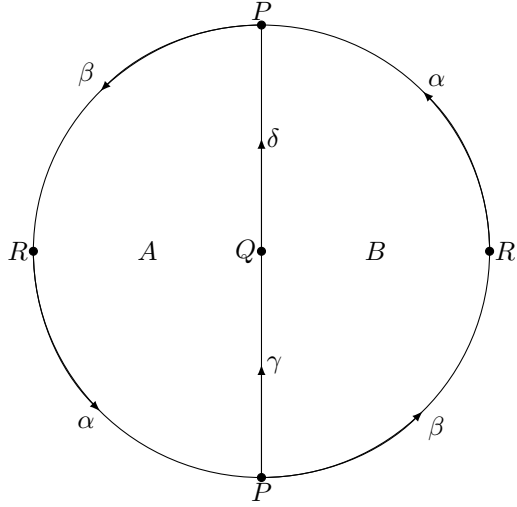

\end{section}


\begin{section}{Attaching maps}\label{sec:attaching}

In this section we will construct attaching maps that exhibit the $\bcell_{\tilde w}$
as the cells of a CW decomposition of $G_k(\RR^n)$.

Let $\tilde w$ be a signed involution and write $l$ for $w(1)$.
Let $\tau_{1,l}$ be  the transposition that switches $1$ and $l$
and decompose $\tilde w = \epsilon_1\tau_{1,l}\cdot \tilde w'$
where $\epsilon_1\in\{\pm1\}$ is a sign and $\tilde w'$ fixes both $e_1$ and $e_l$.

To prepare the ground for an inductive construction of attaching maps we will establish a fibration 
\begin{equation}\label{eq:fiber}
\begin{tikzcd}
F \arrow[r, tail] & \bcell_{\tilde w} \arrow[r, two heads, "p"] & D^{l-1}
\end{tikzcd}
\end{equation}
where $D^{l-1}$ is an open disk of dimension $(l-1)$ and $F\cong\bcell_{\tilde w'}\subset G_{k'}(\RR^{n'})$.
Here 
\begin{equation}\label{eq:cases}
(k',n')=\begin{cases}(k-1,n-2) & l\not=1,\\
(k,n-1) & l=1, \epsilon_1=+1\\
(k-1,n-1) & l=1, \epsilon_1=-1.\end{cases}
\end{equation}

For a matrix $A$ from $BwB$ the $w(1)$ can be recovered as the smallest $j$ such that $Ae_1\in\RR\{e_1,\ldots,e_j\}$.
For $g\in\bcell_{\tilde w}$ this says that $g(e_1)$ is contained in the $(l-1)$-dimensional
half-sphere 
\begin{align}
S^{l-1}_+ &= \setof{(x_1,\ldots,x_l,0,\ldots,0)}{\sum x_i^2=1,\, \epsilon_1\cdot x_l>0}.
\end{align}
This half-sphere is our cell $D^{l-1}$ and the map $p$ is given by $p(g)=g(e_1)$.

We define continuous functions $u_1,u_1^+,u_1^-:D^{l-1}\rightarrow\RR^n$ via
\begin{align}\label{eq:udef}
u_1^\pm(x) &= \frac{e_1\pm x}{\norm{e_1\pm x}}, & u_1(x) &= \epsilon_1 \cdot \frac{\pi_{e_1^\perp}(x)}{\norm{\pi_{e_1^\perp}(x)}}.
\end{align}
Here $\pi_W$ denotes the orthogonal projection onto $W$.
By construction we have
\begin{align}
g(u_1^\pm) &= \pm u_1^\pm,\quad \scprod{u_1}{e_1}=0,\quad  \scprod{u_1}{e_l} > 0.
\end{align}
It's also easy to show that $\vspan\{e_1,u_1\} = \vspan\{u_1^+,u_1^-\}$.

\newcommand{\perpsum}{\mathbin{\oplus}}
Given $g\in \bcell_{\tilde w}$ with $(\pm1)$-eigenspaces $V^\pm$ we now have normalized sections $u_1^\pm \in V^\pm$.
We can use those to decompose the $V^\pm$ as an orthogonal sum: 
\begin{align}
V^\pm &= \RR \cdot u_1^\pm \,\perpsum\, W^\pm, \quad W^\pm = \pi_{\{e_1,u_1\}^\perp}\left(V^\pm\right)
\end{align}

To get a grip on $W^\pm$ we first map it into $\{e_1,e_l\}^\perp$. For this we
quote from \milnor{proof of Lemma 6.3} the transformation $T_{u,v}\in O(n)$ with
\begin{align}\label{eq:tdef}
T_{u,v}(x) &=  x - \frac{\scprod{u+v}{x}}{1+\scprod{u}{v}}\cdot(u+v) + 2\scprod{u}{x}\cdot v.
\end{align}
For unit vectors $u$, $v$ with $u\not=-v$ this formula defines the rotation that maps $u$ to $v$
and leaves everything perpendicular to $u$ and $v$ fixed.

We now let $T = T_{u_1,e_l}$ and find that
\begin{equation}
T\left(W^\pm\right ) = T \left( \pi_{\{e_1,u_1(x)\}^\perp} \left(V^\pm\right) \right) 
= \pi_{\{e_1,e_l\}^\perp} \left( T\left(V^\pm\right)\right) \subset \{e_1,e_l\}^\perp.
\end{equation}

Write $\Phi$ for the composition $T \circ \pi_{\{e_1,u_1(x)\}^\perp} = \pi_{\{e_1,e_l\}^\perp} \circ T$.
Let 
\begin{align}
V^\pm &= \vspan\setof{v_i \pm \epsilon_i v_{w(i)}}{1\le i<w(i)}
\end{align}
be the normal form as in Lemma \ref{l:normform}.
We have
\begin{align}
T\left(W^\pm\right) &= \vspan\setof{\Phi(v_i) \pm \epsilon_i \Phi(v_{w(i)})}{2\le i<w(i)}
\end{align}
since $v_1\pm\epsilon_1v_{w(1)}\in\RR\{e_1,u_1\}$.
The $\Phi(v_i)$ are upper-triangular with respect to the flag 
\begin{align}\label{l:flagdef}
\mathcal{F}_k = \vspan\setof{\Phi(e_j)}{j=2,\ldots,\widehat{l},\ldots,k}.
\end{align}
By Lemma \ref{l:normform} this shows that $T\left(W^\pm\right)$ belongs to the Bruhat cell
$\bcell_{\tilde w'}$ with respect to $\mathcal{F}_\bullet$. 

It remains to exhibit an orthogonal $\Omega\in O(n)$ that maps the flag $\mathcal{F}_\bullet$ 
to the standard flag of $E = \{e_1,e_l\}^\perp$.
Given such an $\Omega$
we can write down a homeomorphism
$\begin{tikzcd}\Xi:D^{l-1}\times \bcell_{\tilde w'}  \arrow[r,"\cong"] & \bcell_{\tilde w}.\end{tikzcd}$
via
\begin{align}\label{eq:homeom}
\Xi(x,R^\pm) &= \RR u_1^\pm + T^{-1}\left(\Omega^{-1}\left(R^\pm\right)\right).
\end{align}

To construct $\Omega$ we must prove that the flag $\mathcal{F}_\bullet$ is indeed well-defined.
First note that $\Phi$ is injective on $E$: let $x\in E$ and decompose $u_1 = B + \lambda e_l$ with $B\in E$ and $\lambda>0$,
$\norm{B}^2+\lambda^2=1$. We find
\begin{align}
\Phi(x) &= x - \frac{\scprod{B}{x}}{1+\lambda}B.
\end{align}
The norm of the second summand is easily seen to be less than $\norm{x}$ which implies $\ker\Phi=0$.

The $\Omega\in O(n)$ is now uniquely determined by the requirements
\begin{equation}
\Omega^{-1}(e_k) \equiv c_k\cdot \Phi(e_k) \mod \setof{\Phi(e_j)}{j<k,\, j\not=l},\quad c_k>0. 
\end{equation}

We have therefore proved
\begin{lemma}
The sequence \eqref{eq:fiber} is a fiber sequence and the map \eqref{eq:homeom}
defines a trivialization 
$\begin{tikzcd}\Xi:D^{l-1}\times \bcell_{\tilde w'}  \arrow[r,"\cong"] & \bcell_{\tilde w}.\end{tikzcd}$
\end{lemma}

We next wish to extend $\Xi$ to the closures of the cells involved in order to get attaching maps for the $\bcell_{\tilde w}$.
Assume inductively we already had an attaching map $\Xi_{\tilde w'}:\widehat{D^{\dim \tilde w'}} \rightarrow \overline{\bcell_{\tilde w'}}$
for the $\tilde w'$-cell.
We will describe below a compactification $\widehat{D^{l-1}}$ of $D^{l-1}$ that still supports the maps
$u_1$, $u_1^\pm$, $T$ and $\Omega$ and which is again topologically an $(l-1)$-cell. 
As in \eqref{eq:homeom} the formula
\begin{align}
\Xi_{\tilde w}(x,y) &= \RR u_1^\pm + T^{-1}\left(\Omega^{-1}\left(\Xi_{\tilde w'}(y)\right)\right)
\end{align}
then defines the required attaching map 
\begin{equation}
\begin{tikzcd}\Xi_{\tilde w}:\widehat{D^{l-1}}\times \widehat{D^{\dim \tilde w'}}  \arrow[r] & \overline{\bcell_{\tilde w}}\end{tikzcd}.
\end{equation}

It remains to describe $\widehat{D^{l-1}}$. 
When we try to extend the maps $u_1$, $u_1^\pm$, $\Omega$ from $D^{l-1}$ to the
naive closure $\overline{D^{l-1}}$ we encounter two problems:
\begin{enumerate}
\item\label{prob:u}
The maps $u_1,u_1^\pm$ become ill-defined at $\pm e_1$.
\item\label{prob:pi}
When $\scprod{u_1}{e_l}$ is allowed to reach zero, the $\Phi(e_i)$ can become zero.
\end{enumerate}

The first problem can be handled by blowing up the points $\pm e_1$ in $\overline{D^{l-1}}$:
let $\overline{D^{l-2}} = \setof{y=(y_2,\ldots,y_l)}{\norm{y}=1,\,\epsilon_1\cdot y_l\ge 0}$ and
$X = [-1;+1]\times \overline{D^{l-2}}$.
Consider the map 
\begin{equation}
\begin{tikzcd}
 X = {}[-1;+1]\times \overline{D^{l-2}}  \arrow[r,"\chi"] & \overline{D^{l-1}}
 \end{tikzcd}
\end{equation}
with $\chi(x_1,y) = \left(x_1,\sqrt{1-x_1^2}\cdot y\right)$.
One finds
\begin{equation}
u_1(\chi(x_1,y)) = y,\quad u_1^\pm(\chi(x_1,y)) = \frac{1}{\sqrt{2}}\left(\sqrt{1\pm x_1},\sqrt{1\mp x_1}\cdot y\right)
\end{equation}
which shows that $u_1$ and $u_1^\pm$ are well-defined and continuous on $X$.

The second problem is solved similarly, by replacing $\overline{D^{l-2}}$ with the product $Y=[0;\pi]^{l-2}$ of
$(l-2)$ closed intervals. Let $\rotat{i/j}{\theta}$ denote the rotation in the plane $\RR\{e_i,e_j\}$ by the angle $\theta$
and let
\begin{align}
R_\theta &= \rotat{l-1/l}{\theta_{l-1}} \circ \cdots \circ \rotat{3/4}{\theta_3} \circ \rotat{2/3}{\theta_2},
\qquad \theta = (\theta_2,\ldots,\theta_{l-1})\in [0;\pi]^{l-2}.
\end{align}
We define $\psi:Y\rightarrow \overline{D^{l-2}}$ via $\psi(\theta) = R_\theta(e_2)$.
Using $c_k = \cos\theta_k$ and $s_k = \sin\theta_k$ this can be described explicitly as
\begin{align}
u_1 &= c_2e_2 + s_2\left( c_3e_3 + s_3 \left( c_4e_4 + \cdots 
+ s_{l-2}\left( c_{l-1} e_{l-1} + s_{l-1} e_l \right) \cdots \right) \right).
\end{align}
We claim that for all $\theta$ the flag $\mathcal{G}_\bullet$ spanned by
\begin{equation}
\pi_{u_1^\perp}(e_2), \pi_{u_1^\perp}(e_3), \ldots, \pi_{u_1^\perp}(e_{l-1})
\end{equation}
is well-defined. Indeed, this follows from the computation
\begin{align}
\pi_{u_1^\perp}(e_2) &= s_2 \cdot R_{\theta}(e_3), \\
\Lambda^2(\pi_{u_1^\perp})(e_2\land e_3) &= s_2 s_3 \cdot \Lambda^2(R_{\theta})(e_3\land e_4), \\
&\cdots \\
\Lambda^{l-1}(\pi_{u_1^\perp})(e_2\land \cdots\land e_{l-1}) &= s_2 \cdots s_{l-1} \cdot \Lambda^{l-1}(R_{\theta})(e_3\land\cdots \land e_l),
\end{align}
which shows $\mathcal{G}_k = \vspan\setof{R_\theta(e_j)}{j=3,\ldots,k+1}$.
It follows that 
\begin{equation}
\mathcal{F}_k = T\mathcal{G}_k = \vspan\setof{TR_\theta(e_j)}{j=3,\ldots,k+1}
\end{equation}
is also well-defined, and the required $\Omega$ can be defined via $\Omega(e_k) = TR_{\theta}(e_{k+1})$.

\end{section}


\makeatletter
\define@key{mp}{posone}[1]{\def\mposone{#1}}
\define@key{mp}{postwo}[2]{\def\mpostwo{#1}}
\define@key{mp}{posthree}[3]{\def\mposthree{#1}}
\define@key{mp}{posfour}[4]{\def\mposfour{#1}}
\define@key{mp}{labelone}[]{\def\mplabone{#1}}
\define@key{mp}{labeltwo}[]{\def\mplabtwo{#1}}
\define@key{mp}{labelthree}[]{\def\mplabthree{#1}}
\define@key{mp}{labelfour}[]{\def\mplabfour{#1}}
\define@key{mp}{hdr}[]{\def\mpheading{#1}}
\makeatother

\newenvironment{matpic}[1]{\setkeys{mp}{hdr,posone,postwo,posthree,posfour,labelone,labeltwo,labelthree,labelfour,#1}
\begin{tikzpicture}[scale=0.5]
\path (-2,-1) rectangle (5,5);
\node [right] at (-2,5.3) {\mpheading};
\draw (-1,0) -- (5,0);
\draw (0,-1) -- (0,4.3);
\ifthenelse{\equal{\mplabone}{}}{}{\draw (\mposone,+0.1) -- (\mposone,-0.1) node [below] {\scriptsize $\mplabone\vphantom{(i)}$};}
\ifthenelse{\equal{\mplabtwo}{}}{}{\draw (\mpostwo,+0.1) -- (\mpostwo,-0.1) node [below] {\scriptsize $\mplabtwo\vphantom{(i)}$};}
\ifthenelse{\equal{\mplabthree}{}}{}{\draw (\mposthree,+0.1) -- (\mposthree,-0.1) node [below] {\scriptsize $\mplabthree\vphantom{(i)}$};}
\ifthenelse{\equal{\mplabfour}{}}{}{\draw (\mposfour,+0.1) -- (\mposfour,-0.1) node [below] {\scriptsize $\mplabfour\vphantom{(i)}$};}
\ifthenelse{\equal{\mplabone}{}}{}{\draw (+0.1,5-\mposone) -- (-0.1,5-\mposone) node [left] {\scriptsize $\mplabone$};}
\ifthenelse{\equal{\mplabtwo}{}}{}{\draw (+0.1,5-\mpostwo) -- (-0.1,5-\mpostwo) node [left] {\scriptsize $\mplabtwo$};}
\ifthenelse{\equal{\mplabthree}{}}{}{\draw (+0.1,5-\mposthree) -- (-0.1,5-\mposthree) node [left] {\scriptsize $\mplabthree$};}
\ifthenelse{\equal{\mplabfour}{}}{}{\draw (+0.1,5-\mposfour) -- (-0.1,5-\mposfour) node [left] {\scriptsize $\mplabfour$};}
}{\end{tikzpicture}}

\begin{section}{The cell complex}\label{sec:incidence}
We now know that the Grassmannian has a CW decomposition indexed by the signed involutions
$\tilde w\in \tilde W$. Our next task is to compute the differential in the associated cellular chain complex. 
This amounts to a determination of the \emph{incidence numbers} $\incid{\tilde v}{\tilde w}$ 
for cells $\bcell_{\tilde v}$, $\bcell_{\tilde w}$ of adjacent dimensions.
We therefore need to understand what the neighborhood of a boundary point 
$\tilde v\in\overline{\bcell_{\tilde w}}$ looks like. 

Classically, when dealing with Schubert cells $X(w)=\overline{BwP} \subset G/P \cong G_k$
the analogous question is well understood:
Assume $v\in X(w)$ and let $B\os\subset G$ denote the 
opposite Bruhat subgroup. Then there is an isomorphism \cite[Lemma 3.2]{MR2422304}
\begin{equation}
\label{eq:schubertproddecomp}
X(w)\cap v B\os P \xrightarrow{\,\cong\,} \mathbb{A}^{l(v)} \times \left(X(w)\cap B\os vP\right) 
\end{equation}
The left hand side is an open neighborhood of $v$ in $X(w)$ and the isomorphism shows that locally near $v$
the variety $X(w)$ decomposes as a product of the affine space $\mathbb{A}^{l(v)}$ and the
{\em Richardson variety} $X_w^v = X(w)\cap B\os vP$. Furthermore, these $X_w^v$ turn out to be
non-singular for a codimension-one boundary point, which implies that near $v$ the cell $X(w)$ 
splits as the product $\RR \times e^l$ where $e^l$ is an open neighborhood of $v$ in $X(v)$.
It follows that $X(v)$ is incident with just two cells of dimension $l(v)+1$ 
corresponding to $\RR_{\ge 0} \times e^l$ and $\RR_{\le 0} \times e^l$. As is well known, these
cells coincide and the incidence numbers are either $0$ or $\pm 2$.

We are going to follow a similar approach here. Let 
$\Invol\subset \Gl(n)$ be the set of all involutions and $\Grassm \subset \Invol$ the subset of
orthogonal involutions.
There is a continuous projection 
\begin{equation}\label{def:orthoproj}
\pi:\Invol \rightarrow \Grassm, \quad
Z \mapsto \theta Z \theta^{-1}
\end{equation}
where $\theta\in B^+$ is defined by the Cholesky decomposition $\theta\os \theta = 1 + Z\os Z$.
One easily checks that $\pi$ is idempotent and that it preserves the Bruhat decomposition given by the
$B^+\tilde vB^+$.

Using $\pi$ we can define a map
\begin{equation}\label{def:kosdecomp}
B \times \left(\Grassm \cap B\os v B\right) \xrightarrow{\,\,\chi\,\,} \Grassm \cap BB\os v B
\end{equation}
via $(b,g) \mapsto \pi\left(bgb^{-1}\right)$. 
In analogy to \eqref{eq:schubertproddecomp} one might conjecture that this map becomes an isomorphism 
when the first factor is suitably restricted. 
This is not true in general, as shown
later on page \pageref{ex:rp2flowlines}.
However, locally near $\tilde v$ the map does induce an isomorphism, and for our purposes that is sufficient.

Recall that $\acell_{\tilde v}$ denotes the space of positive definite, $\tilde v$-invariant matrices that
are normalized as described in Lemma \ref{l:uniquenormform}.
The Cholesky decomposition defines an embedding $\acell_{\tilde v}\subset B^+$ via $\alpha\os\alpha \mapsto\alpha$.

The following Theorem will be proved in subsection \ref{sec:proddec} below.
It serves as our analogue of the isomorphism \ref{eq:schubertproddecomp}.
\begin{thm}\label{thm:kosdecomp}
The restriction of $\chi$ to $\acell_{\tilde v}\times \Grassm$ is a local diffeomorphism near $(1,\tilde v)$.
\end{thm}

The plan for the remainder of this section is this:
We first review the work of Incitti on the covering relations between involutions.
This provides us with a very explicit description of the $\tilde v$ that can occur as a codimension-one
boundary point of a $\bcell_{\tilde w}$.  

We then compute the generalized Richardson varieties 
\begin{equation}\label{def:rhichvar}
\xcell_w^v = \Grassm \cap B\os v B \cap B\tilde w B
\end{equation}
explicitly for these codimension-one points. It turns out that they are circles
that connect $v$, $w$ and two sibling cells $v'$, $w'$ distinct from $v$ and $w$.
This allows us to conclude that $\incid{v}{w}=\pm 1$ for these $v\coveredby w$.

We finally prove Theorem \ref{thm:kosdecomp} and determine the incidence numbers. 


\begin{subsection}{Incitti's classification}
We first need to get a better understanding of the $\tilde v$ that can appear in a given
$\overline{\bcell_{\tilde w}}$.
Starting point are the incidence relations between the $X(w) = \overline{BwB}$, which are well-known:
one has $X(w) \subset X(w')$
if and only if $w \le w'$ in the strict Bruhat order of the permutation group $W$.
The incidence relations between the $\xcell_w = BwB\cap \Grassm$ for involutions $w\in W$
have also been determined:
this is the subject of Incitti's paper \cite{MR2106960}
and we start with a quick recollection of his results.

Consider an involution $w\in W$. The pair $(i,j)$ is called a \emph{rise} of $w$ if $i<j$ and $w(i)<w(j)$.
The rise is \emph{free} if there is no $k$ between $i$ and $j$ with $w(i)<w(k)<w(j)$.

For an integer  $i$ we define the \emph{$w$-type} of $i$, denoted $\wtp{i}$, to be $d$, $e$, or $f$, 
depending on whether $w(i)<i$, or $w(i)=i$, or $w(i)>i$.
The letters $d$, $e$, $f$ stand, respectively, for "deficiency", "excedance" and "fixed point".

The $w$-type of a rise $(i,j)$ is the pair $(\wtp{i},\wtp{j})$. A rise is called \emph{suitable}
if it is free and its type is one of $f\!f$, $fe$, $ef$, $ee$, or $ed$.

A rise of type $ee$ needs to be further differentiated: it is called \emph{crossing} if $i<w(i)<j<w(j)$
and \emph{non-crossing} if $i<j<w(i)<w(j)$.

Finally, recall that for any poset $P$ a \emph{covering relation} $p\coveredby q$
means that $p<q$ and there is no $r$ such that $p<r<q$.

To an involution $w$ and a suitable rise $(i,j)$ Incitti associates a certain \emph{covering transformation}
$\covtrans_{(i,j)}(w)$ such that the covering relations in $\involutions(W)$ are all given
by $w\coveredby\covtrans_{(i,j)}(w)$. We find it convenient to decompose $\covtrans_{(i,j)}(w)$
as the product $w\cdot\covop_{(i,j)}(w)$ and call $\covop_{(i,j)}(w)$ the \emph{covering operation}
associated to the rise $(i,j)$. The $\covop_{(i,j)}(w)$ are then given by 
\begin{equation}
\begin{tabular}{l|l}
type & $\covop_{(i,j)}(w)$ \\
\hline
$f\!f$ & $(i,j)$ \\
$fe$ & $(i,j,w(j))$ \\
$ef$ & $(i,j,w(i))$ \\
$ee$ non-crossing \qquad{} & $(i,j)(w(i),w(j))$ \qquad \\
$ee$ crossing & $(i,j,w(j),w(i))$ \\
$ed$ & $(i,j)(w(i),w(j))$
\end{tabular}
\end{equation}
Here we have employed the usual cycle-notation for permutations where, for example, $(i,j,w(i))$
stands for the permutation with $i\rightarrow j\rightarrow w(i)\rightarrow i$.
\begin{thm}[Incitti]
Let $v$, $w$ be involutions with $\celldim{v} = \celldim{w}-1$. 
Then $v< w$ if and only if there is a suitable rise $(i,j)$ of $v$ such that $w=v\cdot \covop_{(i,j)}(v)$.
The pair $(i,j)$ is uniquely determined by $v$ and $w$.
\end{thm}
\begin{proof}
This is essentially Theorem 5.1 in \cite{MR2106960}. Our cell dimension $\celldim{w}$ is easily 
seen to coincide with the \emph{rank function} $\rho(w)$ of Incitti's Theorem 5.2.
The pair $(i,j)$ can be recovered from $(v,w)$ as the difference- and covering-index
$(di,ci)$, as in \incitti{Section 4}.
\end{proof}

We can now state our signed refinement of Incitti's theorem:
\begin{thm}\label{l:signdet}
Let $\tilde v,\tilde w\in \tilde W$ be signed involutions with 
$\tilde v\in \overline{\bcell_{\tilde w}}$ and $\celldim{\tilde v} = \celldim{\tilde w}-1$.
Write $v$, $w$ for the underlying permutations of $\tilde v, \tilde w$
and let $i<j$ be the suitable rise of $v$ such that $w = v \cdot \covop_{(i,j)}(v)$. 
Then $\tilde w = \tilde v \cdot \covop_{(i,j)}(v) \cdot D$
where $D\in \{\pm1\}^n$ is a matrix of signs and 
the \emph{signed covering operation} 
$\tilde\covop := \covop_{(i,j)}(v) \cdot D = \tilde v\tilde w$
matches the pattern described in Table \ref{tab:covrel}.
\end{thm}

\begin{example}
The simplest kind of covering relation is an $f\!f$-rise $\tilde v\coveredby \tilde w$ where 
$\tilde v$, $\tilde w$ contain respective submatrices
$$V=\begin{pmatrix}\pm 1&\\&\pm 1\end{pmatrix}, \quad 
W=\begin{pmatrix}& \pm 1\\\pm 1&\end{pmatrix}.$$
The assertion of Theorem \ref{l:signdet} in this case is that of the $4\times 4$ conceivable sign combinations
only the $4$ along the circle 
$\left(\begin{smallmatrix}c&\hphantom{+}s\\s&-c\end{smallmatrix}\right)$
with $c^2+s^2=1$ actually occur.
\end{example}

\newcommand{\covtabhdr}{covering operation $\widetilde\covop$\,\,}

\begin{table}
\begin{center}
\begin{tabular}{|l||c||c||c|}
\hline
\rotatebox{90}{\covtabhdr} 
&
$\begin{matpic}{labeltwo=i,labelthree=j,postwo=1.5,posthree=3.5,hdr={$f\!f$-rise}}
\node at (\mpostwo,5-\mposthree) {\scriptsize $+\alpha$};
\node at (\mposthree,5-\mpostwo) {\scriptsize $-\alpha$};
\end{matpic}$
& 
$\begin{matpic}{labeltwo=i,labelthree=j,labelfour=v(j),postwo=1,posthree=2.5,posfour=4,hdr={$fe$-rise}}
\node at (\mposthree,5-\mposfour) {\scriptsize $\alpha$};
\node at (\mpostwo,5-\mposthree) {\scriptsize $\beta$};
\node at (\mposfour,5-\mpostwo) {\scriptsize $\alpha\beta$};
\end{matpic}$
&
$\begin{matpic}{labeltwo=i,labelthree=v(i),labelfour=j,postwo=1,posthree=2.5,posfour=4,hdr={$ef$-rise}}
\node at (\mpostwo,5-\mposfour) {\scriptsize $\alpha$};
\node at (\mposthree,5-\mpostwo) {\scriptsize $\alpha\beta$};
\node at (\mposfour,5-\mposthree) {\scriptsize $\beta$};
\end{matpic}$ \\
\hline
\rotatebox{90}{\covtabhdr} 
&
$\begin{matpic}{labelone=i,labeltwo=j,labelthree=v(i),labelfour=v(j),hdr={$ee$-rise, non-crossing}}
\node at (\mpostwo,5-\mposone) {\scriptsize $-\beta$};
\node at (\mposone,5-\mpostwo) {\scriptsize $+\beta$};
\node at (\mposthree,5-\mposfour) {\scriptsize $+\alpha$};
\node at (\mposfour,5-\mposthree) {\scriptsize $-\alpha$};
\end{matpic}$
&
$\begin{matpic}{labelone=i,labeltwo=v(i),labelthree=j,labelfour=v(j),hdr={$ee$-rise, crossing}}
\node at (\mpostwo,5-\mposone) {\scriptsize $-\alpha\beta\gamma$};
\node at (\mposone,5-\mposthree) {\scriptsize $\gamma$};
\node at (\mposfour,5-\mpostwo) {\scriptsize $\beta$};
\node at (\mposthree,5-\mposfour) {\scriptsize $\alpha$};
\end{matpic}$
&
$\begin{matpic}{labelone=i,labeltwo=v(i),labelthree=v(j),labelfour=j,hdr={$ed$-rise},postwo=1.9,posthree=3.1}
\node at (\mposone,5-\mposfour) {\scriptsize $+\alpha$};
\node at (\mpostwo,5-\mposthree) {\scriptsize $+\beta$};
\node at (\mposthree,5-\mpostwo) {\scriptsize $-\beta$};
\node at (\mposfour,5-\mposone) {\scriptsize $-\alpha$};
\end{matpic}$\\
\hline
\end{tabular}
\end{center}
\caption{Covering relations $\tilde v\coveredby \tilde w$ in $\involutions\left(\tilde W\right)$.
\label{tab:covrel}
The table assumes a suitable rise $(i,j)$ of $\tilde v$
and shows the associated signed covering operation $\widetilde\covop_{(i,j)} = \tilde v\cdot\tilde w$
with the pattern of possible signs $\alpha$, $\beta$, $\gamma$.}
\end{table}

\end{subsection}


\begin{subsection}{The Richardson varieties $\xcell_w^v$}

We now assume a covering relation $\tilde v\coveredby \tilde w$ and want to determine 
$$\xcell_{\tilde w}^{\tilde v} = \Grassm \cap B\os v B \cap B\tilde w B.$$
Starting point is the computation of the corresponding classical Richardson variety 
$X_w^v = B\os vB \cap BwB$ as a subspace of $G/B$.
We first recall from \cite[Theorem 3.1]{MR2761698}
that 
$X_w^v$ only depends on the "interval" $[w,v]$
in the sense that a "pattern embedding" $[x,y] \cong [w,v]$ induces an isomorphism
$X_x^y \cong X_w^v$. 
Incitti's theorem provides us for every rise type with just such 
a pattern embedding from the "model spaces"
$\RR\{e_i,e_j,e_{v(i)},e_{v(j)}\}$ into $\RR^n$.
It follows that the structure of the $X_w^v$ only depends on the rise type.
Furthermore we can compute the $X_w^v$ by just looking at the simplest occurance 
on an $\RR^d$ with $d=2,3,4$.

We follow the general procedure for the computation of a Gröbner basis
for Kazhdan-Lusztig ideals as described in \cite[section 2.2]{MR2956258}
or \cite{MR1154177}.
As in \cite[Theorem 2.1]{MR2956258} this amounts to a determination of 
the {\em essential set} (which gives the generators of a polynomial ring
$P=\RR[a,b,\ldots]$) and the set of {\em essential minors} (which span the definining ideal
$I\subset P$ of $X_w^v = \Spec P/I$).
We have carried this out in Figure \ref{tab:klmatrix} 
for the various $v\coveredby w$. 
In each case we get a description 
$$X_w^v = \setof{Z_{a,b,c,d}\cdot R}{R\in B}$$
for a certain family $Z_{a,b,c,d}$ of matrices. Let 
$$Z_{a,b,c,d} = Q_{a,b,c,d}\cdot R',\qquad Q_{a,b,c,d}\in O(n),\, R'\in B^+$$
be the QR-decomposition. Then
clearly $X_w^v \cap O(n) = \{Q_{a,b,c,d}\}$
and
\begin{equation}
\xcell_w^v = X_w^v \cap O(n) \cap \Invol = \setof{Q_{a,b,c,d}}{Q_{a,b,c,d}^2=1}.
\end{equation}

\begin{thm}
The involutions $Q_{a,b,c,d}\in \xcell_w^v$ are as described in the right column of Figure \ref{tab:klmatrix}. 
\end{thm}

We leave the proof as an exercise to the reader. The computation of the $Q_{a,b,c,d}$ and the identification of the 
involutions among them can easily be done in Sage, for example.




\newcommand{\simpcrossfig}{%
\begin{wrapfigure}{r}{0.4\textwidth} 
  \centering
\begin{tikzpicture}[scale=1.8,
declare function = {
    x(\X,\Y,\Z) = -0.45240233399384144*\X  -0.7754879850219883*\Y  +0.44039813042684567*\Z;
    z(\X,\Y,\Z) = 0.5897222573416534*\X  -0.6305789871656438*\Y  -0.5045768525616449*\Z;
    y(\X,\Y,\Z) = 0.6689990937159702*\X  +0.03144083382620877*\Y   +0.7425979306297072*\Z;
    f(\e) = \e*sqrt((1-\e*\e)/(1+\e*\e));
    g(\e) = sqrt(1-\e*\e-pow(f(\e),2));
},]
\tikzset{every pin edge/.style={draw=black}}
\filldraw[fill=white, ball color=white] (0,0) circle (1);
\draw[scale=1,domain=-1:0.1,variable=\e,samples=\ifdraft{20}{\highresticks},black] plot ({x(+\e,+f(\e),+g(\e))},{y(+\e,+f(\e),+g(\e))});
\draw[scale=1,domain=-1:0.1,variable=\e,samples=\ifdraft{20}{\highresticks},black] plot ({x(-\e,-f(\e),+g(\e))},{y(-\e,-f(\e),+g(\e))});
\draw[scale=1,domain=-1:0.1,variable=\e,samples=\ifdraft{20}{\highresticks},black] plot ({x(+\e,-f(\e),-g(\e))},{y(+\e,-f(\e),-g(\e))});
\draw[scale=1,domain=-1:0.1,variable=\e,samples=\ifdraft{20}{\highresticks},black] plot ({x(-\e,+f(\e),-g(\e))},{y(-\e,+f(\e),-g(\e))});
\filldraw[fill=black, draw=black] ({x(0,0,1)},{y(0,0,1)}) node [pin={[pin distance=0.7em]30:$w$}] {$\bullet$};
\filldraw[fill=black, draw=black] ({x(0,0,-1)},{y(0,0,-1)}) circle (0.02) node [pin={[pin distance=0.7em]200:$w'$}] {$\bullet$};
\filldraw[fill=black, draw=black] ({x(1,0,0)},{y(1,0,0)}) circle (0.02) node [pin={[pin distance=0.7em]130:$v$}] {$\bullet$};
\filldraw[fill=black, draw=black] ({x(-1,0,0)},{y(-1,0,0)}) circle (0.02) node [pin={[pin distance=0.7em]330:$v'$}] {$\bullet$};
\end{tikzpicture}
\end{wrapfigure}}

\simpcrossfig
\stepcounter{thm}
Definition \thethm. 
\manuallabel{def:efgvars}{\thethm}
It turns out that in all cases, 
except the crossing $ee$-variant, the points 
of $\xcell_w^v$ can be parametrized by the circle
$s^2+c^2=1$ (and a discrete choice of signs).
In the case of a crossing $ee$-rise the variety can be parametrized by
variables $(e,f,g)$ in the unit sphere $S^2$ 
subject to the conditions
 $f =\pm \sqrt{\frac{1-e^2}{1+e^2}}\cdot e$
and $f=\epsilon\delta eg$
where $\epsilon,\delta\in\{\pm1\}$ are two additional signs. 
The curve
connects the four involutions
$v,v'=(\pm1,0,0)$ and $w,w'=(0,0,\pm1)$ as shown in the picture on the right.




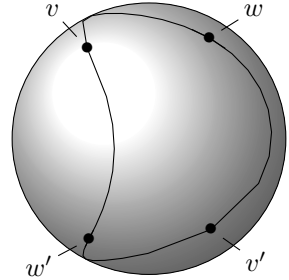
\newcommand{\mpcirc}[2]{\filldraw[fill=white, draw=black] (#1,5-#2) circle (0.1)}
\newcommand{\mpbull}[2]{\filldraw[fill=black, draw=black] (#1,5-#2) circle (0.1)}
\newcommand{\mppath}{\path [draw=black, fill=black!10!white, line width=0.8pt]}
\newlength\stextwidth
\newcommand{\mpvcenter}[1]{\raisebox{-0.5\height}{#1}}
\newcommand{\hdrcolone}{ %
   \tikz[scale=0.5]{\filldraw[fill=black, draw=black] (0,0) circle (0.1)} $=$ $\tilde v$, \, 
   \tikz[scale=0.5]{\filldraw[fill=white, draw=black] (0,0) circle (0.1)} $=$ $\tilde w$ %
}
\def\imagetop#1{\settowidth{\stextwidth}{#1}%
\parbox[c][\height]{\stextwidth}{#1}}
\begin{figure}
\begin{tabular}{|c|>{\centering\arraybackslash} p{3.5cm} <{}|>{\centering\arraybackslash} p{3.6cm} <{}|}
\hline
\hdrcolone & $Z_{a,b,c,d}$ & $\xcell_w^v$ \\
\hhline{|=|=|=|}
\imagetop{$\begin{matpic}{labeltwo=i,labelthree=j,postwo=1.5,posthree=3.5,hdr={$f\!f$-rise}}
\mppath (\mpostwo,5-\mposthree) -- (\mpostwo,5-\mpostwo) -- (\mposthree,5-\mpostwo) -- (\mposthree,5-\mposthree) -- cycle;
\mpcirc{\mpostwo}{\mposthree};
\mpcirc{\mposthree}{\mpostwo};
\mpbull{\mpostwo}{\mpostwo};
\mpbull{\mposthree}{\mposthree};
\end{matpic}$}
&
\imagetop{$\begin{pmatrix*}[r]1&\hphantom{\quad}\\a&1\end{pmatrix*}$}
& \imagetop{$\begin{pmatrix*}[r]c&s\\s&-c\end{pmatrix*}$}
\\
\hline
\imagetop{$\begin{matpic}{labelone=i,labeltwo=j,labelthree=v(j),posone=1,postwo=2.5,posthree=4,hdr={$fe$-rise}}
\mppath (\mposone,5-\mposone) -- (\mposone,5-\mposthree) -- (\mpostwo,5-\mposthree) 
  -- (\mpostwo,5-\mpostwo) -- (\mposthree,5-\mpostwo) -- (\mposthree,5-\mposone) -- cycle;
\mpbull{\mposone}{\mposone};
\mpbull{\mpostwo}{\mposthree};
\mpbull{\mposthree}{\mpostwo};
\mpcirc{\mposone}{\mposthree};
\mpcirc{\mpostwo}{\mpostwo};
\mpcirc{\mposthree}{\mposone};
\end{matpic}$}
&
\imagetop{$\begin{pmatrix*}[r]1&\hphantom{\quad}&\hphantom{\quad}\\a&&1\\b&1&\end{pmatrix*}$}
& \imagetop{$\pm\begin{pmatrix*}[r]c^2 & sc & s\\sc&s^2&-c\\s&-c&\end{pmatrix*}$}
\\
\hline
\imagetop{$\begin{matpic}{labelone=i,labeltwo=v(i),labelthree=j,posone=1,postwo=2.5,posthree=4,hdr={$ef$-rise}}
\mppath (\mposone,5-\mpostwo) -- (\mposone,5-\mposthree) -- (\mposthree,5-\mposthree)
   -- (\mposthree,5-\mposone) -- (\mpostwo,5-\mposone) -- (\mpostwo,5-\mpostwo) -- cycle;
\mpbull{\mposone}{\mpostwo};
\mpbull{\mpostwo}{\mposone};
\mpbull{\mposthree}{\mposthree};
\mpcirc{\mposone}{\mposthree};
\mpcirc{\mpostwo}{\mpostwo};
\mpcirc{\mposthree}{\mposone};
\end{matpic}$}
&
\imagetop{$\begin{pmatrix*}[r]&1&\hphantom{\quad}\\1&\hphantom{\quad}&\\a&b&1\end{pmatrix*}$}
& \imagetop{$\pm\begin{pmatrix*}[r]&-c&s\\-c&s^2&sc\\s&sc&c^2\end{pmatrix*}$}
\\ \hline
\end{tabular}
\caption{Determination of the $X_w^v$ and $\xcell_w^v$ (continued below).}
\end{figure}
\begin{figure}\ContinuedFloat
\begin{tabular}{|c|>{\centering\arraybackslash} p{3.5cm} <{}|>{\centering\arraybackslash} p{3.6cm} <{}|}
\hline
\hdrcolone & $Z_{a,b,c,d}$ & $\xcell_w^v$ \\
\hhline{|=|=|=|}
\imagetop{$\begin{matpic}{labelone=i,labeltwo=j,labelthree=v(i),labelfour=v(j),hdr={$ee$-rise, non-crossing}}
\mppath (\mposone,5-\mposthree) -- (\mposone,5-\mposfour) -- (\mpostwo,5-\mposfour) -- (\mpostwo,5-\mposthree) -- cycle;
\mppath (\mposthree,5-\mposone) -- (\mposfour,5-\mposone) -- (\mposfour,5-\mpostwo) -- (\mposthree,5-\mpostwo) -- cycle;
\mpcirc{\mposone}{\mposfour};
\mpcirc{\mpostwo}{\mposthree};
\mpcirc{\mposthree}{\mpostwo};
\mpcirc{\mposfour}{\mposone};
\mpbull{\mposone}{\mposthree};
\mpbull{\mpostwo}{\mposfour};
\mpbull{\mposthree}{\mposone};
\mpbull{\mposfour}{\mpostwo};
\end{matpic}$}
&
\imagetop{$\begin{pmatrix*}[r]\hphantom{\quad}&\hphantom{\quad}&1&\hphantom{\quad}\\&&a&1\\1&&\hphantom{\quad}&\\b&1&&\end{pmatrix*}$}
& \imagetop{$\begin{pmatrix*}[r]&&s&c\\&&c&-s\\s&c&&\\c&-s&&\end{pmatrix*}$}
\\
\hline
\imagetop{$\begin{matpic}{labelone=i,labeltwo=v(i),labelthree=j,labelfour=v(j),hdr={$ee$-rise, crossing}}
\mppath (\mposone,5-\mpostwo) -- (\mposone,5-\mposfour) -- (\mposthree,5-\mposfour) --(\mposthree,5-\mposthree) 
   -- (\mposfour,5-\mposthree) -- (\mposfour,5-\mposone) -- (\mpostwo,5-\mposone) -- (\mpostwo,5-\mpostwo) -- cycle;
\mpbull{\mposone}{\mpostwo};
\mpbull{\mpostwo}{\mposone};
\mpbull{\mposthree}{\mposfour};
\mpbull{\mposfour}{\mposthree};
\mpcirc{\mposone}{\mposfour};
\mpcirc{\mpostwo}{\mpostwo};
\mpcirc{\mposthree}{\mposthree};
\mpcirc{\mposfour}{\mposone};
\end{matpic}$}
&
\imagetop{$\genfrac{}{}{0pt}{}{\begin{pmatrix*}[r]\hphantom{\quad}&1&%
\hphantom{\quad}&\hphantom{\quad}\\1&\hphantom{\quad}&&\\a&b&&1\\c&d&1&%
\end{pmatrix*}}{\vphantom{\int}\text{(with $ad-bc=0$)}}$}
& \imagetop{\begingroup 
\setlength\arraycolsep{1pt} $\genfrac{}{}{0pt}{}{\begin{pmatrix*}[r]& e & f & g \\ e & \delta t & -\delta ef & -\epsilon f \\%
f & -\delta ef & -\delta t & \epsilon e \\%
g & -\epsilon f & \epsilon e & \end{pmatrix*}}{\text{(with $t=1-e^2$)}}$ \endgroup }

\\
\hline
\imagetop{$\begin{matpic}{labelone=i,labeltwo=v(i),labelthree=v(j),labelfour=j,hdr={$ed$-rise},postwo=1.9,posthree=3.1}
\mppath (\mposone,5-\mpostwo) -- (\mposone,5-\mposthree) -- (\mpostwo,5-\mposthree) --(\mpostwo,5-\mposfour) -- (\mposthree,5-\mposfour) 
   -- (\mposthree,5-\mposthree) -- (\mposfour,5-\mposthree) -- (\mposfour,5-\mpostwo) -- (\mposthree,5-\mpostwo) 
   -- (\mposthree,5-\mposone) -- (\mpostwo,5-\mposone) -- (\mpostwo,5-\mpostwo) -- cycle;
\mpbull{\mposone}{\mpostwo};
\mpbull{\mpostwo}{\mposone};
\mpbull{\mposthree}{\mposfour};
\mpbull{\mposfour}{\mposthree};
\mpcirc{\mposone}{\mposthree};
\mpcirc{\mpostwo}{\mposfour};
\mpcirc{\mposthree}{\mposone};
\mpcirc{\mposfour}{\mpostwo};
\end{matpic}$}
&
\imagetop{$\genfrac{}{}{0pt}{}{\begin{pmatrix*}[r]\hphantom{\quad}&1&\hphantom{\quad}&\hphantom{\quad}%
\\1&\hphantom{\quad}&&\\a&b&&1\\c&d&1&\end{pmatrix*}}{\vphantom{\int}\text{(with $b=c=0$)}}$}
& \imagetop{$\begin{pmatrix*}[r]& c& s& \\ c & & & \mp s \\ s & & & \pm c \\ & \mp s & \pm c &  \end{pmatrix*}$}
\\ \hline
\end{tabular}
\caption{\label{tab:klmatrix}%
Determination of the $X_w^v$ and $\xcell_w^v$.}\end{figure}

\end{subsection}


\begin{subsection}{The product decomposition}
\label{sec:proddec}
We now turn to the proof of Theorem \ref{thm:kosdecomp}. The map in question
\begin{equation}\label{eq:chiproddec2}
\chi: \acell_{\tilde v} \times 
\left(\Grassm \cap B\os v B \right) 
\longrightarrow \Grassm 
\end{equation}
is given by 
$\chi \left( \alpha\os\alpha , g \right)  = \theta \alpha g \alpha^{-1} \theta^{-1}$
where $\theta\os\theta = 1+ \alpha^{-\ast} g \alpha\os \alpha g_t \alpha$.
Using $\theta = \psi\alpha$ and $A=\alpha\os\alpha$ we can write this as
\begin{align}
\chi \left( A , g \right) &= \psi g \psi^{-1}, \qquad \psi\os\psi = A + g A g.
\end{align}

We want to show that $\Grassm \cap B\os v B$ is a manifold of dimension complementary
to $\celldim{\tilde v}$.
Let $\widehat{\rcell_{\tilde v}} = \setof{A=A\os>0}{A^{\tilde v}=A^{-1}}$.
\begin{lemma}\label{l:rnormal}
$\Grassm \cap B\os v B = \setof{b\os vb}{b\in B^+,\, bb\os\in\widehat{\rcell_{\tilde v}}}$.
\end{lemma}
\begin{proof}
The proof is similar to the proof of Lemma \ref{l:alphadef}.
First note that $B\os v B = w_0Bw_0vB$ is a translated Bruhat cell (where $w_0$ is the maximal involution),
so every $g\in B\os v B$
has a unique representation as $g=b\os v c$ with $b\in U\cap vUv$, $c\in B^+$.
If furthermore $g\in O(n)$ then $c$ is determined by $b$.
Now let $g\in \Grassm \cap B\os v B$.
We decompose $c=c_{\textrm{inv}}\cdot c_{\textrm{rise}}$ where
$c_{\textrm{inv}} \in U\cap vU\os v$, $c_{\textrm{rise}}\in B\cap vBv$.
Taking the conjugate gives
$g = c_{\textrm{rise}}\os v \cdot v c_{\textrm{inv}}\os v \cdot b$.
By uniqueness we have $g = b\os v d b$ where $d=c_{\textrm{inv}}$ satisfies $d^v = d\os$.
In $U\cap vU\os v$ the binomial formula provides us with a square root $s$ of $d$ which also
satisfies $s^v = s\os$. Then
$g = b\os v ss b = b\os vsv vsb = b\os s\os v sb = (sb)\os v (sb)$, as desired.
We leave the remaining verifications to the reader.
\end{proof}

It follows that 
$\Grassm \cap B\os v B \cong \widehat{\rcell_{\tilde v}}/{\sim}$
where the identifications are given by $A\sim \beta A\beta\os$ 
whenever $\beta\in B^+$ with $\beta^v = \beta^{-\ast}$. Let $\rcell_{\tilde v}$ denote the
quotient $\widehat{\rcell_{\tilde v}}/{\sim}$.
We claim that it is a manifold.

To see this, first note that $\langle x,y\rangle_{\tilde v} = \langle \tilde vx,y\rangle$ defines a scalar product
of signature $(p_+,p_-)$ where $p_\pm = \dim E^\pm$ with the $(\pm1)$-eigenspaces 
$E^\pm = \setof{x}{\tilde vx=\pm x}$ of $\tilde v$.
The associated orthogonal group is $O_{\tilde v} = \setof{g\in\Gl_n}{g^{\tilde v}=g^{-1}}$ and one has
$$\widehat{\rcell_{\tilde v}} = \bigslant{O_{\tilde v}}{O(E^+)\times O(E^-)}.$$
Let  $H=\setof{\beta\in B^+}{\beta^{\tilde v} = \beta^{-\ast}}$. One has $H = O_{\tilde v} \cap B^+$, so
$$\rcell_{\tilde v} = \bigslant{O_{\tilde v}}{\left(O(E^+)\times O(E^-)\right) \cdot H}$$
It follows that $\rcell_{\tilde v}$ is a manifold, as claimed.

One easily checks that the dimension of $\rcell_{\tilde v}$ is given by the number of equivalence classes 
$$\qrises{v} = \bigslant{\setof{(i,j)}{i<j,\,v(i)<v(j)}}{(i,j)\sim (v(i),v(j))}.$$

\begin{lemma}\label{l:chilocaldiff}
The differential of $\chi$ at $(1,\tilde v)$ is an isomorphism.
\end{lemma}
\begin{proof}
Let $t$ be an infinitesimal variable with $t^2=0$.
We consider a first order path $B_t$ in $\rcell_{\tilde v}$ given by $B_t = bb\os = (1+t\eta)(1+t\eta\os)$.
Here $\eta$ is an upper triangular matrix that is supported on the $v$-rises and obeys $\eta^{\tilde v} = -\eta$.
The corresponding path in $\Grassm \cap B\os v B$ is given by
$g_t=b\os v b = v + tv(\eta-\eta\os)$.

Likewise, let the path $A_t$ in $\acell_{\tilde v}$ be represented by 
$A_t = \alpha\os\alpha = (1+t\rho\os)(1+t\rho)= 1+t(\rho + \rho\os)$
where $\rho$ is an upper triangular matrix that is supported on the $v$-inversions
and obeys $\rho^{\tilde v} = \rho\os$. 

One finds $\psi\os\psi = A_t + g_tA_tg_t \equiv 2A_t$, so we get $\psi = \sqrt{2}\alpha$, which gives
$$\chi(A_t,g_t) = (1+t\rho) g_t (1-t\rho) = v + t\cdot v\cdot\left( \eta-\eta\os + \rho\os - \rho \right).$$
The differential thus represents the decomposition of a 
$\phi\in T_{\tilde v}\Grassm = \setof{\phi\in\lie{so}_n}{\tilde v \phi = -\phi\tilde v}$
as the sum of its $v$-inversion part $\rho\os-\rho$ and its $v$-rising part $\eta-\eta\os$. 
\end{proof}

We now know that there is a neighborhood $U\times V\subset \acell_{\tilde v}\times (\Grassm\cap B\os vB)$ that is mapped
diffeomorphically onto a neigborhood $W=\chi(U\times V)\subset\Grassm$. 
Upon intersecting this with the Bruhat cell $BwB$ we obtain the desired isomorphism
$U\times (V\cap \xcell_w^v) \cong W\cap \Grassm \cap \overline{BwB}$.

\newcommand{\rpfl}[1]{%
\draw[scale=1,domain=1:-0.1,variable=\c,samples=\ifdraft{\lowresticks}{\highresticks},black!40!white] plot ({2+2*x(\c,#1)},{2+2*y(\c,#1)});
\draw[scale=1,domain=1:-0.1,variable=\c,samples=\ifdraft{\lowresticks}{\highresticks},black!40!white] plot ({2-2*x(\c,#1)},{2+2*y(\c,#1)});
\draw[scale=1,domain=-1:0.1,variable=\c,samples=\ifdraft{\lowresticks}{\highresticks},black!40!white] plot ({2+2*x(\c,#1)},{2+2*y(\c,#1)});
\draw[scale=1,domain=-1:0.1,variable=\c,samples=\ifdraft{\lowresticks}{\highresticks},black!40!white] plot ({2-2*x(\c,#1)},{2+2*y(\c,#1)});
}
\newcommand{\avarline}[1]{%
\draw[scale=1,domain=-0.9:0.9,variable=\a,samples=\ifdraft{\lowresticks}{\highresticks},black!40!white] plot ({2+2*x3(\a,#1,sqrt(1-(#1*#1))))},{2+2*y3(\a,#1,sqrt(1-(#1*#1))))});
\draw[scale=1,domain=-0.9:0.9,variable=\a,samples=\ifdraft{\lowresticks}{\highresticks},black!40!white] plot ({2-2*x3(\a,#1,sqrt(1-(#1*#1))))},{2+2*y3(\a,#1,sqrt(1-(#1*#1))))});
} 
\newcommand{\exrptwofig}{%
\begin{tikzpicture}[framed,scale=0.9,
declare function={ B(\c,\a) = pow(\c,2)/(1+\a*\c*(1-pow(\c,2)));
                   A(\c,\a) = (\c+\a*(1-pow(\c,2))/2)/sqrt(1+\a*\c*(1-pow(\c,2)));
                   X(\b) = -1*sqrt(1-\b)/sqrt(2);
                   Y(\a,\b) = \a*sqrt((1-\b)/(1-pow(\a,2)))/sqrt(2);
                   x(\c,\a) = X(B(\c,\a));
                   y2(\c,\a) = Y(A(\c,\a),B(\c,\a));
                   y(\c,\a) = A(\c,\a)/(sqrt(2) * sqrt(1-(\a*\c+pow(\a,2)*(1-pow(\c,2))/4)/(1+\a*\c)));
                   pf(\a,\c,\s) = (\c+\a*\s*\s/2)/sqrt(1+\a*\c*\s*\s));
                   qf(\a,\c,\s) = \c*\c/(1+\a*\c*\s*\s);
                   pc1(\a,\c,\s) = sqrt(1-pow(pf(\a,\c,\s),2));
                   pc2(\a,\c,\s) = -pf(\a,\c,\s);
                   pc3(\a,\c,\s) = sqrt( (1-pow(pf(\a,\c,\s),2)) * (1+qf(\a,\c,\s))/(1-qf(\a,\c,\s)) - pow(pf(\a,\c,\s),2)) );
                   pcnrm(\a,\c,\s) = sqrt(pow(pc1(\a,\c,\s),2) + pow(pc2(\a,\c,\s),2) + pow(pc3(\a,\c,\s),2)); 
                   x3(\a,\c,\s) = pc1(\a,\c,\s) / pcnrm(\a,\c,\s) ; 
                   y3(\a,\c,\s) = pc2(\a,\c,\s) / pcnrm(\a,\c,\s) ;
                 },
]
\foreach \z in {-0.999,-0.8,-0.6,-0.4,-0.2,0,0.2,0.4,0.6,0.8,0.999} \rpfl{\z};
\foreach \c in {-0.99,-0.9,...,1} \avarline{\c};
\draw (2,2) circle [radius=2];
\draw[thin] (4,2)  ;
\draw[thin] (2,4)  ;
\draw[thin] (0,2)  ;
\draw[thin] (2,0)  ;
\draw[thin] (2,0) -- (2,1) ;
\draw[thin] (2,2) -- (2,3) ;
\draw (2,0) -- (2,4);
\draw[white!30!red, ultra thick] (2,2) circle [radius=sqrt(2)];
\draw[white!60!blue, ultra thick, -round cap] (2,4) -- (2,2);
\draw[fill] (2,3.414213562373095) circle [radius=1pt];
\draw[fill] (0.5857864376269051,2) circle [radius=1pt] ;
\draw[fill] (2,0) circle [radius=1pt]  node[below] {$P$};
\draw[fill] (2,4) circle [radius=1pt]  node[above] {$P$};
\draw[fill] (2,2) circle [radius=1pt]  node[left]  {$R$};
\draw[fill] (4,2) circle [radius=1pt]  node[right] {$Q$};
\draw[fill] (0,2) circle [radius=1pt]  node[left]  {$Q$};

\end{tikzpicture}%
}
\newcommand{\exwrfig}{%
\begin{wrapfigure}{r}{0.4\textwidth}
\centering
\exrptwofig
\end{wrapfigure}%
}
 
\begin{example}
\label{ex:rp2flowlines}
The following picture illustrates the map $\chi$ in case of 
the Bruhat decomposition of $\RR P^2$ (compare Figure \ref{fig:rp2example}).
It shows the parallel projection of its universal covering $S^2$ into the $(x,y)$-plane.
We have chosen
\begingroup
\renewcommand*{\arraystretch}{0.7}
\begin{equation}
\tilde v = \begin{pmatrix*}1&&\\&&1\\&1&\end{pmatrix*},
\quad
\tilde w = \begin{pmatrix*}&&1\\&1&\\1&&\end{pmatrix*}.
\end{equation}
\endgroup
The Richardson 
variety $\xcell_w^v$ is
the red circle through $\tilde v$ and $\tilde w$. The cell $\bcell_{\tilde v}$ is indicated in blue. 
The grey lines show the images $\chi(\alpha,\ast)$ and $\chi(\ast,g)$ 
for some constant
values of $A = \alpha\os\alpha =\left(\begin{smallmatrix}1&&\\&1&a\\&a&1\end{smallmatrix}\right)$
and $g\in\xcell_w^v$. 
One can see that $\chi$ is neither surjective nor injective.
\end{example}  
\exwrfig
The underlying computation assumes $g\in\xcell_w^v$ as in the $fe$-case of Figure \ref{tab:klmatrix}. 
With 
$\gamma = \left(\begin{smallmatrix}1&c&c^2\\&s&cs\\&&s\end{smallmatrix}\right)$
one has $g=\gamma w\gamma^{-1}$.
We determine the $w$-coordinates $C = \beta\os\beta$ of $\chi(A,g) = \psi g\psi^{-1} = \beta w \beta^{-1}$
via
$C = \beta\os\beta = \gamma\os\left( A + gA g\right)\gamma$.
One has $C = 2\cdot\left(\begin{smallmatrix}k&l&m\\l&n&l\\m&l&k\end{smallmatrix}\right)$
with
$k=1+acs^2$, $l=c+\frac12as^2$, $m=c^2$, $n=1$. 
After normalisation we get
\begingroup
\renewcommand*{\arraystretch}{0.7}
$$\beta\os\beta\sim \begin{pmatrix*}1&p&q\\p&1&p\\q&p&1\end{pmatrix*}
\quad\text{with}\quad
p=\frac{c+\frac12as^2}{\sqrt{1+acs^2}},\quad q = \frac{c^2}{1+acs^2}.$$
\endgroup
The Cholesky decomposition gives
$$\beta = \begin{pmatrix*}
1 & p & q \\
0 & \sqrt{-p^{2} + 1} & -\frac{p {\left(q - 1\right)}}{\sqrt{-p^{2} + 1}} \\
0 & 0 & \sqrt{\frac{p^{2} {\left(q - 1\right)}^{2}}{p^{2} - 1} - q^{2} + 1}
\end{pmatrix*},$$
so the represented ray $\RR\cdot \beta(e_1-e_3) \in \RR P^2$ has the homogeneous coordinates
\begin{multline}
\left[
q-1
:
-\frac{p {\left(q - 1\right)}}{\sqrt{-p^{2} + 1}} 
:
\sqrt{\frac{p^{2} {\left(q - 1\right)}^{2}}{p^{2} - 1} - q^{2} + 1}
\,\right]
\\ =
\left[
\sqrt{1-p^2}
:
-p 
:
\sqrt{(1-p^2)\frac{1+q}{1-q} - p^2}
\,\right]
\end{multline}

\end{subsection}

\begin{subsection}{Incidence numbers}
Let $\tilde v\coveredby \tilde w$ be a covering relation as in the previous sections.
Recall the notation $\tilde v(e_j) = \epsilon_j\cdot e_{v(j)}$ which separates
the underlying permutation $v$ and the vector of signs $\epsilon$.
In this section the $\epsilon_j$ will always refer to the signs of a $\tilde v$
(not the $\tilde w$ which is also present).

We have seen that the Richardson variety $\xcell_w^v$ is topologically a circle
that connects $\tilde w$, $\tilde v$ and two ``sibling cells'' $\tilde w'$ and $\tilde v'$.
It follows that there is a path $t\mapsto g_t$ within $\xcell_w^v$ from $g_0=\tilde w$ to 
$g_1=\tilde v$. 
We will compute an explicit choice of such a path and use it
to propagate the orientation of $\bcell_{\tilde w}$ to $\tilde v$.
Comparing the resulting orientation to the one inherited from $\bcell_{\tilde v}$
(together with an inward-pointing tangent vector at $\tilde v$)
allows us to deduce the sign of the incidence number
$\incid{\tilde v}{\tilde w} = \pm1$.

As above, our path will be parametrized using the circle $s^2+c^2=1$, except in the crossing $ee$ case
where we use the $(e,f,g)$-curve from Definition  \ref{def:efgvars}.
In the circle case we 
let $s=\sqrt{1-t^2}$ (assumed to be non-negative) and $c=\theta t$ with an appropriate sign $\theta\in\{\pm1\}$.

In the crossing $ee$-case we let $D=\epsilon_q\gamma$, $E=-\alpha\beta\gamma \epsilon_p$
and define $e=tD\beta$, $f=\pm\sqrt{\frac{1-e^2}{1+e^2}}$, $g=D\sqrt{1-t^2}$.
We choose the sign of $f$ such that $eg+D\alpha\beta f=0$.
Here $\alpha$, $\beta$, $\gamma$ are the parameters of the rise as in Table \ref{tab:covrel}.
\begin{lemma}\label{def:lambda}
Let $\theta = \epsilon_p\epsilon_q$ for an $fe$-rise, $\theta = -\epsilon_p\epsilon_q$ for a 
non-crossing $ee$-rise and $\theta=1$ otherwise.
Define $g_t = \lambda_t \tilde w\lambda_t^{-1}$ with $\lambda_t$ as in the following table.
Then $g_t$ is a path in $\xcell_w^v$ from $\tilde w$ to $\tilde v$.
\begin{center}
\begin{tabular}{|b{3.4cm}|b{3.5cm}|b{3.4cm}|}
\hline
$f\!f$-rise
&
$fe$-rise 
&
$ef$-rise
\\
\begin{flushright}$\begin{pmatrix*}[r]1 & -\alpha c\\ & s\end{pmatrix*}$\end{flushright}
&
\begin{flushright}$\begin{pmatrix*}[r]1 & -c\beta & c^2\alpha\beta\\&s&-sc\\&&s
\end{pmatrix*}$\end{flushright}
&
\begin{flushright}$\begin{pmatrix*}[r]1 & c\alpha\beta & \\ & 1 & c\beta \\ &&s\end{pmatrix*}$\end{flushright}
\\
\hline
$ee$ non-crossing
&
$ee$ crossing
&
$ed$-rise
\\
\begin{flushright}$\begin{pmatrix*}[r]1&-c\alpha\\&s\\&&1&-c\beta\\&&&s\end{pmatrix*}$\end{flushright}
&
\begin{flushright}$\begin{pmatrix*}[r]1&Ee&\frac{Ef}{e^2-1}\\&1&\frac{-ef}{e^2-1}&De\\&&1&Df\\&&&Dg\end{pmatrix*}$\end{flushright}
&
\begin{flushright}$\begin{pmatrix*}[r]1&&&-\alpha c\\&1&-c\beta\\&&s\\&&&s\end{pmatrix*}$\end{flushright}
\\
\hline
\end{tabular}
\end{center}
\end{lemma}
The proof, of course, is a straightforward computation that will be ommited.

As already mentioned, we can use these paths to propagate the orientation of the cell
from the center $\tilde w$ to a boundary point $\tilde v$.
This calculation can be carried out numerically, and with sufficient precision it
leads to an exact determination of the incidence number. We record the results
of such a computation in the following Lemma.

Recall that for each rise type there is a minimal dimension where it can occur.
We call these the ``model rises'' (or ``model coverings'')
since a pattern embedding reduces a general $\tilde v\coveredby \tilde w$
to one of these cases.

\begin{lemma}\label{comp:incidmodel}
Let $\tilde v\coveredby \tilde w$ be a model rise of type $r$, realized in dimension $d$,
and let $\alpha$, $\beta$, $\gamma$ be the parametrization as in Table \ref{tab:covrel}.
Then the incidence number $\incid{\tilde v}{\tilde w}$ between
$\bcell_{\tilde w}$ and $\bcell_{\tilde v}$ is as follows:
\begin{center}
\begin{tabular}{|l|l|l|}
\hline
$\incid{\tilde v}{\tilde w}$ & rise type & $d$ 
\\
\hline
$-\alpha$ & $f\!f$ & 2
\\$1$& $fe$ & 3
\\$-\alpha$&  $ef$ & 3
\\$\alpha \beta$&  $ee$ non-crossing & 4
\\$1$& $ee$ crossing & 4
\\$-\alpha$& $ed$ & 4
\\
\hline
\end{tabular}
\end{center}
\end{lemma}
\begin{proof}
Machine verified.
\end{proof}

We next show how to reduce the computation of a general $\incid{\tilde v}{\tilde w}$ to
one of these model-computations.

Let $i,j$ be as in Table \ref{tab:covrel}
and define $D=\{i,j,v(i),v(j)\}$, $D^c = \setof{k}{1\le k\le n,\, k\not\in D}$.
Writing 
$$\tilde x = \tilde v\vert_D,\quad \tilde u = \tilde w\vert_D,\quad \tilde r = \tilde v\vert_{D^c} = \tilde w\vert_{D^c}$$
gives decompositions 
\begin{equation}\label{comp:redmod}
\tilde v = \tilde r \oplus \tilde x,\quad \tilde w = \tilde r \oplus \tilde u
\end{equation}
where $\tilde x\coveredby\tilde u$ is one of the model rises and $\tilde r$ is common to both $\tilde v$ and $\tilde w$.

Recall (see the discussion following Definition \ref{acelldef})
that the tangent space $T_{\tilde w}\bcell_{\tilde w}$ 
has a basis given by the set 
$$\qinversions{\tilde w} = \setof{(i,j)}{i<j,\,w(i)>w(j)} / (i,j)\simeq (w(j),w(i)).$$
It follows that we can write
\begin{align}
\label{eq:tangdecomp}
T_{\tilde w}\bcell_{\tilde w}
&=
\underbrace{\,\RR\, \qinversions{\tilde w\vert_D}\,}_{=: W_{DD}}
\, \oplus \,
\underbrace{\,\RR\, \qinversions{\tilde w\vert_{D^c}}\,}_{=: W_{D^cD^c}}
\, \oplus \,
\underbrace{\,\RR \left( \inversions{\tilde w} \cap \left(D\times D^c\right)\right)\,}_{=: W_{DD^c}}.
\end{align}
This decomposition is realized by a shuffle permutation of $\qinversions{\tilde w}$ that we denote $\sigma_{\tilde w}$.

There is a corresponding decomposition of $T_{\tilde v}\bcell_{\tilde v}$ and
we can compare these pieces and their contribution to $\incid{\tilde v}{\tilde w}$ one-by-one:
\begin{center}
\begin{tabular}{cl}
$W_{DD}$ vs.{} $V_{DD}$ & isomorphic to the model rise, contributes $\incid{\tilde x}{\tilde u}$. \\
$W_{D^cD^c}$ vs.{} $V_{D^cD^c}$ & these are identical, no contribution. \\
$W_{DD^c}$ vs. $V_{DD^c}$ & related by an isomorphism that is induced by \\ &a path from $\tilde w$ to $\tilde v$,
thus contributes an extra \\& sign $\xi(\tilde v,\tilde w)$.
\end{tabular}
\end{center}
Together these observations imply the
\begin{lemma}\label{comp:incid}
One has $\incid{\tilde v}{\tilde w} = \incid{\tilde x}{\tilde u} \cdot \sign\sigma_{\tilde w} 
\cdot \sign \sigma_{\tilde v} \cdot \xi(\tilde v,\tilde w)$.
\end{lemma}

We close this section with a more explicit description of the sign $\xi(\tilde v,\tilde w)$.

Assume a path $g=g(t)$ in $\bcell_{\tilde w}$ between $\tilde v$ and $\tilde w$, with $\dot g\not=0$ everywhere.
we can think of $T_g\bcell_{\tilde w}$ as a set of
symmetric, $g$-invariant matrices $C$; indeed,
using the isomorphism
$$\bcell_{\tilde w} = \{\alpha g\alpha^{-1}\} 
\,\leftrightarrow\, 
\left(\bigslant{\setof{A=\alpha\os\alpha}{A = A\os = A^g > 0}}{C(g)\cap B}\right)$$
we have
$$T_g \bcell_{\tilde w} = \bigslant{\setof{C}{C = C\os = C^g}}{\left(C\simeq C+\eta+\eta\os, \eta\in C(g)\cap B\right)}.$$
There is then a map $\phi_g:T_g\bcell_{\tilde w} \rightarrow T_{\tilde v}\bcell_{\tilde v} \oplus \RR$
given by 
\begin{equation}
C \mapsto \left(\tilde vgC + Cg\tilde v,\scprod{C}{\dot g}\right)
\end{equation}
We believe that this map is an isomorphism as long as the rotation $\tilde vg$ does not
map any non-trivial $x$ to a perpendicular vector $\tilde vg(x)$.
Indeed, under this condition the map $C\mapsto \tilde vg C + C\tilde vg$ is invertible. 
This precludes the interesting case $g=\tilde w$, though, so it does not lead to a direct
computation of $\incid{\tilde v}{\tilde w}$.
To compute $\xi(\tilde v,\tilde w)$, however, 
we only need the restriction $\phi_g^{DD^c}:W_{DD^c}\rightarrow V_{DD^c}$,
and this turns out to be well-behaved.

We now assume $g(t) = \lambda \tilde w\lambda^{-1}$ where $\lambda$
is the map from Lemma \ref{def:lambda} (with an appropriate parametrisation $\lambda = \lambda(t)$). 
This choice
guarantees that $\phi_g$ respects
the decomposition \eqref{eq:tangdecomp}.
\begin{lemma}
On $W_{DD^c}$ the map $\phi^{DD^c}_{\tilde w}$ is given by $C\mapsto \tilde\sigma C + C\tilde\sigma^{-1}$
where $\tilde\sigma = \tilde v\tilde w$ is the signed covering operation.
This map is an isomorphism and induces the extra sign $\xi(\tilde v,\tilde w)$.
\end{lemma} 
\begin{proof}
One easily checks that $\phi^{D^cD^c}_g$ is the identity. 
On $W_{DD}$ the map $\phi$ might develop singularities near $\tilde w$, but we
can replace it by any convenient continuos identification 
$T_g\bcell_{\tilde u} \cong T_v\bcell_{\tilde x}\oplus\RR$ without changing 
$\phi^{DD^c}_g$. So if we can show that the latter is an isomorphism
it will automatically qualify for the computation of $\incid{\tilde v}{\tilde w}$.

It remains to investigate the behaviour on $W_{DD^c}$. 
Let $S_{i,j}$ denote the symmetric matrix $X$ with $X_{i,j} = X_{j,i} = 1$ and $X_{p,q}=0$ for other $(p,q)$.
Write
$$\tilde w(e_i) = \delta_i e_{w(i)},\quad
  \tilde v(e_i) = \epsilon_i e_{v(i)},\quad
  \tilde \sigma(e_i) = s_i e_{\sigma(i)}.$$
One has $s_i = \epsilon_{w(i)}\cdot\delta_i$.
With 
$C^v_{i,j} = S_{i,j} + \epsilon_i\epsilon_j C_{v(i),v(j)}$,
$C^w_{i,j} = S_{i,j} + \delta_i\delta_j C_{w(i),w(j)}$
we find that $V_{DD^c}$ and $W_{DD^c}$ have the bases
\begin{align}
\setof{C^v_{i,j}}{i\in D,\, j\in D^c,\, (i<j \land v(i)>v(j)) \lor (i>j \land v(i)<v(j))} &\subset V^{DD^c}, \\
\setof{C^w_{i,j}}{i\in D,\, j\in D^c,\, (i<j \land w(i)>w(j)) \lor (i>j \land w(i)<w(j))} & \subset W^{DD^c}.
\end{align}
A straightforward computation (using $v(j)=w(j)$ since $j\in D^c$) 
shows that $\phi$ maps
$$C^w_{i,j} \mapsto 
S_{i,j} + \epsilon_{w(i)}\delta_i S_{\sigma(i),j}
+ \delta_i\delta_j S_{w(i),w(j)} + \delta_j\epsilon_i S_{v(i),v(j)}.$$
By assumption, $\tilde\sigma$ matches one of the patterns of Table \ref{tab:covrel};
this is seen to imply $\delta_j = \epsilon_{w(j)}$. We thus find
\begin{equation}
\phi: C^w_{i,j} \mapsto \delta_j\epsilon_i C^v_{v(i),w(j)} + \delta_i\delta_j C^v_{w(i),w(j)}
\end{equation}
To show that $\phi$ is invertible we can thus take $j$ and $w(j)$ to be fixed.
We are left with the map $e_i\mapsto \epsilon_i e_{v(i)} + \delta_i e_w(i)$
from the model space $\RR^d$ to itself. Its invertibility can then be checked by hand.
\end{proof}

Note that an enumeration of the bases of $T_{\tilde v}\bcell_{\tilde v}$
and $T_{\tilde w}\bcell_{\tilde w}$ provides us with a second identification $W_{DD^c}\cong V_{DD^c}$.
It thus actually makes sense to speak of the determinant $\det \phi^{DD^c}_{\tilde w} = \xi(\tilde v,\tilde w)$.

\end{subsection}


\begin{subsection}{Oriented Grassmannians}

For a signed involution $\tilde v$ the corresponding eigenspace $E_{\tilde v}^-$ has the basis
$\setof{e_i-\tilde v(e_i)}{1\le i\le v(i)}$. 
Putting this in ascending order with respect to $i$ gives us a preferred orientation,
which we denote $\tilde v^+$. The opposite orientation is then denoted $\tilde v^-$. 
The $\tilde v^{\pm}$ index the cells of a CW-decomposition of the oriented Grassmannian $G_k\oriented(n)$.

For a rise $\tilde v\coveredby \tilde w$ one can ask whether the orientation of
$E_{\tilde v}^-$ induced from $\tilde w^+$ agrees with $\tilde v^+$ or $\tilde v^-$.
We let $\orid{\tilde v}{\tilde w} = +1$ in the first case, $-1$ in the second.
Knowledge of $\orid{\tilde v}{\tilde w}$ allows to deduce the incidence numbers between the cells
of the oriented Grassmannian $G_k\oriented(\RR^n)$.

We first compute the $\orid{\tilde v}{\tilde w}$ for the model rises.
\begin{lemma}\label{comp:ormodel}
Let $\tilde v\coveredby \tilde w$ be a model rise of type $r$ and dimension $d$.
Let $i$, $j$, $\alpha$, $\beta$, $\gamma$ be as in Table \ref{tab:covrel}
and write $\tilde v(e_k) = \epsilon_k\cdot e_{v(k)}$.
Then $\orid{\tilde v}{\tilde w}=+1$ except for the following cases:
\begin{center}
\begin{tabular}{|l|l|l|}
\hline
rise type & condition & $d$ 
\\
\hline
$f\!f$ & $(\epsilon_i,\alpha) = (+1,-1)$ & 2
\\$fe$ & $(\epsilon_i,\beta) = (+1,-1)$ & 3
\\$ef$ & $(\epsilon_j,\beta) = (-1,+1)$ & 3
\\$ee$ crossing & $\alpha\beta=-1,\,\gamma\epsilon_i=-1$ & 4
\\$ed$ & $\beta=-1$ & 4
\\
\hline
\end{tabular}
\end{center}
In particular, a non-crossing model $ee$-rise always has $\orid{\tilde v}{\tilde w}=+1$.
\end{lemma}
\begin{proof}
The $\lambda$ from Lemma \ref{def:lambda} provides us with an explicit path from $\tilde w$ to $\tilde v$
and a straightforward computation (easily implemented in Sage, for example) 
allows us to compare the resulting orientations at $\tilde v$.  
The details are left to the machine.
\end{proof}

For a general covering relation $\tilde v\coveredby \tilde w$
the computation of $\orid{\tilde v}{\tilde w}$ can be reduced to the model case
as follows: 
let $\tilde v = \tilde r \oplus \tilde x$, $\tilde w = \tilde r \oplus \tilde u$
be the decomposition as in \eqref{comp:redmod}.
Let 
\begin{align}
I_{\tilde v} &= \setof{i}{1\le i\le n,\, 
(i<v(i)) \lor \left( \left(i=v\left(i\right)\right) \land \left(\tilde v\left(e_i\right)=-e_i\right)\right)} 
\end{align}
be the index set for $E_{\tilde v}^- = \RR\setof{e_i-\tilde v(e_i)}{i\in I_{\tilde v}}$.
The decomposition
$$I_{\tilde v} = \left( I_{\tilde v} \cap D \right) \amalg  \left( I_{\tilde v} \cap D^c \right)$$
is realized by a shuffle permutation $\rho_{D,\tilde v}$.
\begin{lemma}\label{comp:orid}
One has $\orid{\tilde v}{\tilde w} = \orid{\tilde x}{\tilde u} \cdot \sign(\rho_{D,\tilde v}) \cdot \sign(\rho_{D,\tilde w})$.
\end{lemma}

The proof is left to the reader.

\end{subsection}

\end{section}


\newenvironment{gtfenvdimension}[1]
    {
    
    \pagebreak[3]
    Dimension #1\nopagebreak[4]
    \begin{center}
    \begin{tabular}{p{0.16\textwidth}p{0.7\textwidth}}
    }
    { 
    \end{tabular} 
    \end{center}
    }
\newcommand{\gtfenvii}[3]{\hfill$#1\mapsto$&$#2#3$}
\newcommand{\gtfenviiii}[5]{\hfill$#1\mapsto$&$#2#3#4#5$}
\newcommand{\gtfenviiiiii}[7]{\hfill$#1\mapsto$&$#2#3#4#5$\\&$#6#7$}
\newcommand{\gtfenviiiiiiii}[9]{\hfill$#1\mapsto$&$#2#3#4#5$\\&$#6#7#8#9$}
\newcommand{\gtfenvlinebreak}{\\}

\begin{section}{Examples}\label{sec:examples}
We use cycle notation $(p_1q_1)\cdots(p_kq_k)$ to denote the involution with $p_j\leftrightarrow q_j$.
An underlined cycle will indicate an additional sign flip, 
e.g.{} $(\underline{pq})$ interchanges $e_p \leftrightarrow -e_q$.

The formulas we provide will apply to the oriented Grassmannians $G_k\oriented(\RR^n)$; for $G_k(\RR^n)$
one just ignores the exponents.

\begin{subsection}{Cell structure of $\RR P^\infty$}
The cells of $\RR P^{n-1}$ are 
given by the $(pq)$, $(\underline{pq})$
for $1\le p<q\le n$ and the $(\underline{p})$ for $1\le p\le n$.
The differential is given by
\begin{align}\label{rpndiff}
(p,q)^{\epsilon} &\mapsto 
(p+1,q)^{\epsilon\hphantom{-}}
+ (\underline{p+1,q})^{-\epsilon}
- (p,q-1)^{\epsilon}
- (\underline{p,q-1})^{\epsilon}
\\
(\underline{p,q})^{\epsilon} &\mapsto 
(p+1,q)^{-\epsilon}
+ (\underline{p+1,q})^{\epsilon\hphantom{-}}
+ (p,q-1)^{\epsilon}
+ (\underline{p,q-1})^{\epsilon}
\end{align}
For a $1$-cell $(p,p+1)$ this simplifies to
$d\left( (p,p+1)^\epsilon \right) = (\underline{p+1})^{-\epsilon} - (\underline{p})^{\epsilon}$
and
$d\left( (\underline{p,p+1})^\epsilon \right) = -(\underline{p+1})^{-\epsilon} + (\underline{p})^{\epsilon}$.

One can introduce an ad-hoc coproduct $\Delta$ via
\begin{align}
\Delta (pq)^\epsilon &= \sum_{p<r<q} (pr)^\epsilon\otimes (rq)^\epsilon + (\underline{pr})^\epsilon \otimes (\underline{rq})^\epsilon,
\\
\Delta (\underline{pq})^\epsilon &= \sum_{p<r<q} (\underline{pr})^\epsilon\otimes (rq)^\epsilon 
+ (pr)^\epsilon \otimes (\underline{rq})^\epsilon.
\end{align}
The differential is then $\Delta$-comultiplicative.

The following tables describe the differential explicitly for $\RR P^4$.

\begin{gtfenvdimension}{4}
\gtfenviiii%
{(15)^+}%
{-(14)^+}%
{-(\underline{14})^+}%
{+(25)^+}%
{+(\underline{25})^-}%
\gtfenvlinebreak
\gtfenviiii%
{(\underline{15})^+}%
{+(14)^+}%
{+(\underline{14})^+}%
{+(25)^-}%
{+(\underline{25})^+}%
\end{gtfenvdimension}

\begin{gtfenvdimension}{3}
\gtfenviiii%
{(14)^+}%
{-(13)^+}%
{-(\underline{13})^+}%
{-(24)^+}%
{+(\underline{24})^-}%
\gtfenvlinebreak
\gtfenviiii%
{(\underline{14})^+}%
{+(13)^+}%
{+(\underline{13})^+}%
{+(24)^-}%
{-(\underline{24})^+}%
\gtfenvlinebreak
\gtfenviiii%
{(25)^+}%
{-(24)^+}%
{-(\underline{24})^+}%
{-(35)^+}%
{+(\underline{35})^-}%
\gtfenvlinebreak
\gtfenviiii%
{(\underline{25})^+}%
{+(24)^+}%
{+(\underline{24})^+}%
{+(35)^-}%
{-(\underline{35})^+}%
\end{gtfenvdimension}

\begin{gtfenvdimension}{2}
\gtfenviiii%
{(13)^+}%
{-(12)^+}%
{-(\underline{12})^+}%
{+(23)^+}%
{+(\underline{23})^-}%
\gtfenvlinebreak
\gtfenviiii%
{(\underline{13})^+}%
{+(12)^+}%
{+(\underline{12})^+}%
{+(23)^-}%
{+(\underline{23})^+}%
\gtfenvlinebreak
\gtfenviiii%
{(24)^+}%
{-(23)^+}%
{-(\underline{23})^+}%
{+(34)^+}%
{+(\underline{34})^-}%
\gtfenvlinebreak
\gtfenviiii%
{(\underline{24})^+}%
{+(23)^+}%
{+(\underline{23})^+}%
{+(34)^-}%
{+(\underline{34})^+}%
\gtfenvlinebreak
\gtfenviiii%
{(35)^+}%
{-(34)^+}%
{-(\underline{34})^+}%
{+(45)^+}%
{+(\underline{45})^-}%
\gtfenvlinebreak
\gtfenviiii%
{(\underline{35})^+}%
{+(34)^+}%
{+(\underline{34})^+}%
{+(45)^-}%
{+(\underline{45})^+}%
\end{gtfenvdimension}

\begin{gtfenvdimension}{1}
\gtfenvii%
{(12)^+}%
{+(\underline{2})^-}%
{-(\underline{1})^+}%
\gtfenvlinebreak
\gtfenvii%
{(\underline{12})^+}%
{-(\underline{2})^+}%
{+(\underline{1})^+}%
\gtfenvlinebreak
\gtfenvii%
{(23)^+}%
{+(\underline{3})^-}%
{-(\underline{2})^+}%
\gtfenvlinebreak
\gtfenvii%
{(\underline{23})^+}%
{-(\underline{3})^+}%
{+(\underline{2})^+}%
\gtfenvlinebreak
\gtfenvii%
{(34)^+}%
{+(\underline{4})^-}%
{-(\underline{3})^+}%
\gtfenvlinebreak
\gtfenvii%
{(\underline{34})^+}%
{-(\underline{4})^+}%
{+(\underline{3})^+}%
\gtfenvlinebreak
\gtfenvii%
{(45)^+}%
{+(\underline{5})^-}%
{-(\underline{4})^+}%
\gtfenvlinebreak
\gtfenvii%
{(\underline{45})^+}%
{-(\underline{5})^+}%
{+(\underline{4})^+}%
\end{gtfenvdimension}

\end{subsection}
\begin{subsection}{Cell structure of $G_2(\RR^4)$}
\quad \\

\begin{gtfenvdimension}{4}
\gtfenviiii%
{(14)(23)^+}%
{+(14)(\underline{3})^-}%
{-(14)(\underline{2})^+}%
{+(13)(\underline{24})^+}%
{+(\underline{13})(24)^+}%
\gtfenvlinebreak
\gtfenviiii%
{(14)(\underline{23})^+}%
{-(14)(\underline{3})^+}%
{+(14)(\underline{2})^+}%
{-(13)(24)^+}%
{-(\underline{13})(\underline{24})^+}%
\gtfenvlinebreak
\gtfenviiii%
{(\underline{14})(23)^+}%
{+(\underline{14})(\underline{3})^-}%
{-(\underline{14})(\underline{2})^+}%
{-(13)(24)^+}%
{-(\underline{13})(\underline{24})^+}%
\gtfenvlinebreak
\gtfenviiii%
{(\underline{14})(\underline{23})^+}%
{-(\underline{14})(\underline{3})^+}%
{+(\underline{14})(\underline{2})^+}%
{+(13)(\underline{24})^+}%
{+(\underline{13})(24)^+}%
\end{gtfenvdimension}

\begin{gtfenvdimension}{3}
\gtfenviiiiiiii%
{(14)(\underline{3})^+}%
{+(13)(\underline{4})^-}%
{+(\underline{13})(\underline{4})^+}%
{+(12)(34)^+}%
{+(12)(\underline{34})^+}%
{+(\underline{12})(34)^+}%
{+(\underline{12})(\underline{34})^+}%
{-(24)(\underline{3})^+}%
{+(\underline{24})(\underline{3})^-}%
\gtfenvlinebreak
\gtfenviiiiiiii%
{(14)(\underline{2})^+}%
{-(13)(\underline{2})^+}%
{-(\underline{13})(\underline{2})^+}%
{+(12)(34)^+}%
{+(12)(\underline{34})^-}%
{+(\underline{12})(34)^-}%
{+(\underline{12})(\underline{34})^+}%
{-(\underline{1})(24)^+}%
{+(\underline{1})(\underline{24})^+}%
\gtfenvlinebreak
\gtfenviiiiiiii%
{(\underline{14})(\underline{3})^+}%
{-(13)(\underline{4})^+}%
{-(\underline{13})(\underline{4})^-}%
{+(12)(34)^+}%
{+(12)(\underline{34})^+}%
{+(\underline{12})(34)^+}%
{+(\underline{12})(\underline{34})^+}%
{+(24)(\underline{3})^-}%
{-(\underline{24})(\underline{3})^+}%
\gtfenvlinebreak
\gtfenviiiiiiii%
{(\underline{14})(\underline{2})^+}%
{+(13)(\underline{2})^+}%
{+(\underline{13})(\underline{2})^+}%
{+(12)(34)^-}%
{+(12)(\underline{34})^+}%
{+(\underline{12})(34)^+}%
{+(\underline{12})(\underline{34})^-}%
{+(\underline{1})(24)^+}%
{-(\underline{1})(\underline{24})^+}%
\gtfenvlinebreak
\gtfenviiiiii%
{(13)(24)^+}%
{-(13)(\underline{4})^-}%
{-(13)(\underline{2})^+}%
{+(12)(\underline{34})^-}%
{-(\underline{12})(34)^+}%
{+(24)(\underline{3})^+}%
{-(\underline{1})(24)^+}%
\gtfenvlinebreak
\gtfenviiiiii%
{(13)(\underline{24})^+}%
{-(13)(\underline{4})^+}%
{-(13)(\underline{2})^+}%
{-(12)(34)^-}%
{+(\underline{12})(\underline{34})^+}%
{-(\underline{24})(\underline{3})^+}%
{+(\underline{1})(\underline{24})^+}%
\gtfenvlinebreak
\gtfenviiiiii%
{(\underline{13})(24)^+}%
{-(\underline{13})(\underline{4})^-}%
{-(\underline{13})(\underline{2})^+}%
{+(12)(34)^+}%
{-(\underline{12})(\underline{34})^-}%
{+(24)(\underline{3})^-}%
{-(\underline{1})(24)^+}%
\gtfenvlinebreak
\gtfenviiiiii%
{(\underline{13})(\underline{24})^+}%
{-(\underline{13})(\underline{4})^+}%
{-(\underline{13})(\underline{2})^+}%
{-(12)(\underline{34})^+}%
{+(\underline{12})(34)^-}%
{-(\underline{24})(\underline{3})^-}%
{+(\underline{1})(\underline{24})^+}%
\end{gtfenvdimension}

\begin{gtfenvdimension}{2}
\gtfenviiii%
{(13)(\underline{4})^+}%
{-(12)(\underline{4})^+}%
{-(\underline{12})(\underline{4})^+}%
{+(23)(\underline{4})^+}%
{+(\underline{23})(\underline{4})^-}%
\gtfenvlinebreak
\gtfenviiii%
{(13)(\underline{2})^+}%
{+(12)(\underline{3})^-}%
{+(\underline{12})(\underline{3})^+}%
{+(\underline{1})(23)^+}%
{+(\underline{1})(\underline{23})^+}%
\gtfenvlinebreak
\gtfenviiii%
{(\underline{13})(\underline{4})^+}%
{+(12)(\underline{4})^+}%
{+(\underline{12})(\underline{4})^+}%
{+(23)(\underline{4})^-}%
{+(\underline{23})(\underline{4})^+}%
\gtfenvlinebreak
\gtfenviiii%
{(\underline{13})(\underline{2})^+}%
{-(12)(\underline{3})^+}%
{-(\underline{12})(\underline{3})^-}%
{+(\underline{1})(23)^+}%
{+(\underline{1})(\underline{23})^+}%
\gtfenvlinebreak
\gtfenviiii%
{(12)(34)^+}%
{+(12)(\underline{4})^-}%
{-(12)(\underline{3})^+}%
{-(\underline{2})(34)^-}%
{+(\underline{1})(34)^+}%
\gtfenvlinebreak
\gtfenviiii%
{(12)(\underline{34})^+}%
{-(12)(\underline{4})^+}%
{+(12)(\underline{3})^+}%
{-(\underline{2})(\underline{34})^-}%
{+(\underline{1})(\underline{34})^+}%
\gtfenvlinebreak
\gtfenviiii%
{(\underline{12})(34)^+}%
{+(\underline{12})(\underline{4})^-}%
{-(\underline{12})(\underline{3})^+}%
{+(\underline{2})(34)^+}%
{-(\underline{1})(34)^+}%
\gtfenvlinebreak
\gtfenviiii%
{(\underline{12})(\underline{34})^+}%
{-(\underline{12})(\underline{4})^+}%
{+(\underline{12})(\underline{3})^+}%
{+(\underline{2})(\underline{34})^+}%
{-(\underline{1})(\underline{34})^+}%
\gtfenvlinebreak
\gtfenviiii%
{(24)(\underline{3})^+}%
{+(23)(\underline{4})^-}%
{+(\underline{23})(\underline{4})^+}%
{+(\underline{2})(34)^+}%
{+(\underline{2})(\underline{34})^+}%
\gtfenvlinebreak
\gtfenviiii%
{(\underline{24})(\underline{3})^+}%
{-(23)(\underline{4})^+}%
{-(\underline{23})(\underline{4})^-}%
{+(\underline{2})(34)^+}%
{+(\underline{2})(\underline{34})^+}%
\gtfenvlinebreak
\gtfenviiii%
{(\underline{1})(24)^+}%
{-(\underline{1})(23)^+}%
{-(\underline{1})(\underline{23})^+}%
{+(\underline{1})(34)^+}%
{+(\underline{1})(\underline{34})^-}%
\gtfenvlinebreak
\gtfenviiii%
{(\underline{1})(\underline{24})^+}%
{+(\underline{1})(23)^+}%
{+(\underline{1})(\underline{23})^+}%
{+(\underline{1})(34)^-}%
{+(\underline{1})(\underline{34})^+}%
\end{gtfenvdimension}

\begin{gtfenvdimension}{1}
\gtfenvii%
{(12)(\underline{4})^+}%
{+(\underline{2})(\underline{4})^-}%
{-(\underline{1})(\underline{4})^+}%
\gtfenvlinebreak
\gtfenvii%
{(12)(\underline{3})^+}%
{+(\underline{2})(\underline{3})^-}%
{-(\underline{1})(\underline{3})^+}%
\gtfenvlinebreak
\gtfenvii%
{(\underline{12})(\underline{4})^+}%
{-(\underline{2})(\underline{4})^+}%
{+(\underline{1})(\underline{4})^+}%
\gtfenvlinebreak
\gtfenvii%
{(\underline{12})(\underline{3})^+}%
{-(\underline{2})(\underline{3})^+}%
{+(\underline{1})(\underline{3})^+}%
\gtfenvlinebreak
\gtfenvii%
{(23)(\underline{4})^+}%
{+(\underline{3})(\underline{4})^-}%
{-(\underline{2})(\underline{4})^+}%
\gtfenvlinebreak
\gtfenvii%
{(\underline{23})(\underline{4})^+}%
{-(\underline{3})(\underline{4})^+}%
{+(\underline{2})(\underline{4})^+}%
\gtfenvlinebreak
\gtfenvii%
{(\underline{1})(23)^+}%
{+(\underline{1})(\underline{3})^-}%
{-(\underline{1})(\underline{2})^+}%
\gtfenvlinebreak
\gtfenvii%
{(\underline{1})(\underline{23})^+}%
{-(\underline{1})(\underline{3})^+}%
{+(\underline{1})(\underline{2})^+}%
\gtfenvlinebreak
\gtfenvii%
{(\underline{2})(34)^+}%
{+(\underline{2})(\underline{4})^-}%
{-(\underline{2})(\underline{3})^+}%
\gtfenvlinebreak
\gtfenvii%
{(\underline{2})(\underline{34})^+}%
{-(\underline{2})(\underline{4})^+}%
{+(\underline{2})(\underline{3})^+}%
\gtfenvlinebreak
\gtfenvii%
{(\underline{1})(34)^+}%
{+(\underline{1})(\underline{4})^-}%
{-(\underline{1})(\underline{3})^+}%
\gtfenvlinebreak
\gtfenvii%
{(\underline{1})(\underline{34})^+}%
{-(\underline{1})(\underline{4})^+}%
{+(\underline{1})(\underline{3})^+}%
\end{gtfenvdimension}

\end{subsection}
\end{section}


\bibliography{grassmann}

\end{document}